\documentclass[11pt]{article}
\usepackage{graphicx,wrapfig,lipsum}
\usepackage{amsmath,amssymb,graphicx} 
\usepackage{tcolorbox}
\usepackage{amsfonts}
\usepackage{pifont}
\usepackage{dsfont}
\usepackage{mathtools}
\usepackage{extarrows}
\usepackage{float}
\usepackage{amsthm}
\usepackage{enumerate}
\usepackage[toc]{appendix}
\usepackage{hyperref}
\usepackage{fullpage}
\usepackage{tikz}
\usepackage{tcolorbox}
\usepackage{bbm}
\usepackage{comment}
\newtheorem{definition}{Definition}[section]
\newtheorem{theorem}{Theorem}

\newtheorem{lemma}{Lemma}
\newtheorem{remark}{Remark}
\usepackage{comment}
\usepackage{enumitem}
\usepackage{stackengine}
\usepackage[utf8]{inputenc}
\usepackage{xcolor}
\usepackage{graphicx}
\usepackage[normalem]{ulem}
\graphicspath{ {./images/} }
\allowdisplaybreaks

\allowdisplaybreaks
\def\es{\varepsilon}
\def\orr{\Omega_R^\varepsilon}
\def\omm{\Omega_M^\varepsilon}
\def\oll{\Omega_L^\varepsilon}
\def\oii{\Omega_i^\varepsilon}

\def\mor{\Omega_R}
\def\mol{\Omega_L}
\def\p{\partial}

\def\mint{\int_{0}^{T}\int_{\Sigma}\int_{\mathcal{T}}\int_{Y}\int_{Z}}
\def\Mint{\int_0^T\int_{\omm}}
\def\dzdy{\di{z}\di{y}\di{\tau}\di{\overline{x}}\di{t}}
\def\T{\mathcal{T}}
\usepackage{color}
\usepackage{cleveref}
\usetikzlibrary{arrows,decorations.pathmorphing}
\newcommand{\di}[1]{\,\mathrm{d}#1}

\begin{document}
	
	
	\title{Interface Conditions for Wave Propagation Through a Time-varying Metasurface}

\author{Vishnu Raveendran$^a$\footnote{raveendr@ins.uni-bonn.de}\,  and Barbara Verfürth$^a$ \\ 
$^a$Institute for Numerical Simulation, University of Bonn, Germany.}

	\date{} 
	\maketitle
	
	\begin{abstract}\label{abstract}
    
	 We study wave propagation through a time-modulated thin heterogeneous layer.  The layer is assumed to have a thickness  of order $\es\ll 1$ and its material properties exhibit rapid oscillations on multiple spatial and temporal scales.
We aim to rigorously derive the effective model and the corresponding interface conditions for wave propagation through the limiting interface. 
 
To perform the homogenization together with dimension reduction, we generalize the notion of two-scale convergence for thin layer introduced by Neuss-Radu and Jäger (2007) to a multiple space-time scale framework. One of the main analytical difficulties arising in the homogenization analysis is that, in general, a uniform energy estimate cannot be obtained for a wave equation with time-varying coefficients. Therefore, we identify two physically relevant classes of coefficients for which we can derive a uniform energy bound. These include coefficients with traveling wave-type modulations of their properties.

Using the energy estimates and the multiscale convergence for thin layer concept, we derive the effective model, which  consists of linear wave equations in the bulk domains coupled through a nonstandard jump condition at the interface. This jump condition is governed by a wave-type dynamical equation with effective macroscopic coefficients determined by suitable cell problems of elliptic and hyperbolic type.  Finally, we discuss the uniqueness of the effective model. 
     
	\end{abstract}
		{\bf Keywords}: Homogenization; metamaterial; time modulated wave equation; dimension reduction; thin layer; multiscale convergence; reiterated homogenization.
		\\
		{\bf MSC2020}: 35B27; 78A40; 35L05; 78-10.    
\maketitle

\section{Introduction}\label{section:Introduction}
A metasurface is an artificially engineered ultra-thin surface made up of periodically arranged substructures whose size is much smaller than the wavelength to manipulate wave propagation. 
Depending on the physical situation, a metasurface  can be used to control transmission, reflection, phase separation, or the direction of wave propagation as they arise from electromagnetic, acoustic, or elastic settings \cite{Davies2025Roadmap}. In recent years, increasing attention is given to adding time modulation to such metasurfaces (or metamaterials) to generate dynamical effects that are generally not attainable by time-independent metasurfaces \cite{HuidobroSilveirinhaGaliffiPendry2021,Shaltout2019Spatiotemporal,DarcheAssierGuenneauLombardTouboul2025,Garg2024TwoStep,AmmariCaoHiltunenRueff2026}. Mathematically, this leads to a wave propagation model with both space-time oscillating coefficients.

In this paper, we are interested in rigorously finding the effective model and the corresponding effective interface conditions of wave propagation through a time-varying heterogeneous layer. We consider two bulk domains separated by a thin heterogeneous layer of thickness $\es$, and assume a wave is propagating from one bulk domain to another through this layer.  In these domains, we consider the wave equation of the type
\begin{align*}
    \frac{\p }{\p t}\left(\rho \frac{\p u}{\p t}\right)-\mathrm{div} \left(D\nabla u\right)=f
\end{align*}
with suitable scalings, initial and boundary conditions. Here the coefficients $\rho$ and $D$ determine the material properties. 
We assume the material properties in the layer can rapidly oscillate with respect to space as well as time such that $\rho$ and $D$ are highly oscillating space-time variables in the layer. We aim to determine the effective behavior of the heterogeneous layer in the limit $\es\rightarrow 0$ and to derive the resulting interface conditions that describe the influence of heterogeneity and temporal modulation of the layer for wave propagation.

Since the layer has a thickness of order $\es$, to let $\es\rightarrow 0$ in the microscopic model and derive the upscaled model, we require a homogenization framework that accounts for homogenization with multiple scales in space and time as well as dimension reduction. In the same spirit as the classical two-scale convergence \cite{Allaire1992} generalized to multi-scale convergence \cite{allaire1996multiscale}, we generalize the concept of two-scale convergence for thin layer \cite{neuss2007effective} to a setting involving multiple spatial and temporal scales. This leads us to the notion of multiscale convergence for thin layer.

To use the multiscale convergence for the thin layer, we require uniform energy estimates for our microscopic model. In contrast to a static medium, a time-varying medium may transfer energy to the propagating wave, leading to a blow-up of the wave energy. Consequently, in general, we cannot obtain a uniform energy bound.  Therefore, we characterize two physically relevant structures of the coefficient for which we get a uniform energy bound. Then, using the energy estimates and  multiscale convergence for the thin layer, we derive the effective model and the associated interface conditions along the limiting interface. We observe that the first interface condition gives the continuity of the solution along the effective interface. But in the second interface condition, which describes the jump of the normal flux across the interface, we obtain a dynamical wave-type equation whose coefficients are determined by a suitable hyperbolic and elliptic cell problem.  In particular, the hyperbolic structure of the cell problem arises due to the time modulation of the microscopic problem.

Homogenization problems related to wave propagation in spatially heterogeneous media are widely studied \cite{Santosa1991Dispersive,BrahimOtsmane1992Correctors,doding2026optimizationbased,AllaireConca1998}. However, finding the effective model for wave propagation in time-modulated media is significantly less explored, even for a regular domain. One of the main reasons is that the time modulation can cause instability in the solution \cite{colombini1984hyperbolic}. Formal homogenization of such problems is addressed in \cite{DV26,TouboulLombardAssierGuenneauCraster2024}, while a rigorous derivation is studied in \cite{casado2011homogenization}, in which the author considers strong spatial oscillation and weak temporal oscillation.

To the best of our knowledge, no previous work has addressed the derivation of interface conditions via homogenization and dimension reduction for a time-modulated thin layer. Nevertheless, homogenization and dimension reduction techniques have been successfully applied to find effective models for a variety of physical problems \cite{BuffordDavoliFonseca2015,Freudenberg2026Analysis,FatimaIjiomaOgawaMuntean2014,Raveendran2022Scaling,MarusicPalokaPazanin2026}. In particular, spatially oscillating wave equations are considered in \cite{Gahn2026Effective,Gahn2026TwoScale} and related problems for Maxwell's equations have been analyzed in \cite{SchweizerWiedemann2026,AiyappanGrisoOrlikSufian2026}.

We organize the paper as follows: In \Cref{sect:microscopic_model}, we first introduce the microscopic model along with the microscopic geometry. Then we present the suitable assumptions on the data and the weak formulation of the problem together with the well-posedness result. In \Cref{section:Uniform energy estimates and compactness results}, we discuss under what additional conditions on the data the uniform energy estimates are guaranteed and provide a discussion of why we cannot expect a uniform energy bound in a general temporal oscillation setting. Later, we introduce the multiscale convergence for thin layer and derive the related compactness result. We also establish the compactness result on the bulk as well as on the interfaces connecting the bulk and the layer. Finally, in \Cref{section:Upscaling of the microscopic model}, we derive the effective model and interface condition, in addition, we study the uniqueness of the effective model.

\section{Microscopic model}\label{sect:microscopic_model}

Let $\ell,h, T, \es>0$, and $\Omega\subset\mathbb{R}^2$ be a rectangular domain defined as $(-\ell,\ell)\times (0,h)$ consisting of three subdomains and their interfaces (see Fig. \ref{micmod}). We denote these subdomains by $\Omega_L^\es, \Omega_M^\es$ and $\Omega_R^\es$ in which $\Omega_M^\es$ contains a heterogeneous structure. We denote the interface between $\Omega_L^\es$ and $\Omega_M^\es$ as $L^\es$, and  the interface between $\Omega_M^\es$ and $\Omega_R^\es$ as $R^\es$.
Precisely, the subdomains and interfaces are given by the following definition
\begin{align*}
    \oll := \left(-\ell,\dfrac{-\es}{2}\right)\times (0,h),\qquad
    \orr := \left(\dfrac{\es}{2},\ell\right)\times (0,h),\qquad
    \omm :=\left(\dfrac{-\es}{2},\dfrac{\es}{2}\right)\times (0,h)
    \end{align*}
    \begin{align*}
    L^\es:= \left\{\left(\dfrac{-\es}{2},x_2\right): x_2\in (0,h)\right\},\qquad
    R^\es:= \left\{\left(\dfrac{\es}{2},x_2\right): x_2\in (0,h)\right\}.
\end{align*}
Therefore, we have $\Omega=\oll\cup L^\es\cup\omm\cup R^\es\cup\orr$.  For simplicity, we assume $h=N\es$ for some $N\in \mathbb{N}$.

Now to describe the heterogeneous structure of $\omm$, we first define the standard cell $Y$ as a unit square $(\tfrac{-1}{2},\tfrac{1}{2})\times (0,1)$  made up of materials that may change their properties with respect to time, for example we can consider the unit square $(\tfrac{-1}{2},\tfrac{1}{2})\times (0,1)$ with a time-modulated material located at the center (see Fig. \ref{scell}). We denote the left vertical boundary of $Y$ as $Y_L$ and the right vertical boundary as $Y_R$.
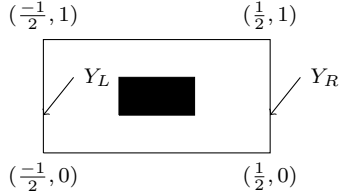
\begin{figure}[ht]
	\begin{center}
		\begin{tikzpicture}
		\draw (0,0) node [anchor=north] {{\scriptsize $(\frac{-1}{2},0)$}} to (3,0) node [anchor=north] {{\scriptsize $(\frac1 2,0)$}}  to (3,1.5)node [anchor=south] {{\scriptsize $(\frac 12,1)$}} to (0,1.5)node [anchor=south] {{\scriptsize $(\frac{-1}{2},1)$}} to (0,0);
		\draw [fill] (1,.5) to (2,.5) to (2,1) to (1,1) to (1,.5);
		\draw[<-](0,0.5) to (0.4,1) node[anchor=west] {{\scriptsize $Y_{L}$}};
		\draw[<-](3,0.5) to (3.4,1) node[anchor=west] {{\scriptsize $Y_{R}$}};
		\end{tikzpicture}
		\caption{Schematic representation of a standard cell $Y$ exhibiting a time-varying rectangular material placed in the center.}
		\label{scell}
	\end{center}
\end{figure}

The heterogeneous layer $\omm$ can be described by periodically arranged in one direction $\varepsilon$-scaled versions of $Y$, that is
\begin{align*}
    \omm= \bigcup_{k=0}^{N-1} \es\left((0,k)+Y\right).
\end{align*}

To describe the macroscopic geometry that will be utilized after homogenization, we have the following definition of macroscopic subdomains
 \begin{align*}\Omega_L := \left(-\ell,0\right)\times (0,h),\qquad
    \Omega_R :=\left(0,\ell\right)\times (0,h),
    \end{align*}
and the effective interface
\begin{align*}
    \Sigma :=\{(0,x):x\in (0,h)\},
\end{align*}
as illustrated in Figure \ref{macmod}.

To reduce notational complexity, for any function $g:\Omega\rightarrow \mathbb{R}$,  we write $g_L^\es$ as restriction of $g$ on $\oll$, $g_M^\es$ as restriction of $g$ on $\omm$ and $g_R^\es$ as restriction of $g$ on $\orr$.

\begin{figure}[H] 
		\begin{center}
			\begin{tikzpicture}[scale=1.15,
  >=stealth',
  pos=.8,
  photon/.style={decorate,decoration={snake,post length=1mm}}
]\draw (0,0) node [anchor=north] {} to (6,0) node [anchor=north] {}  to (6,3.75) to (0,3.75) to (0,0);
			\draw (2.5,0)node [anchor=north] {}   to (2.5,3.75)  ;
			\draw (2.5,0)node [anchor=north] {{\scriptsize $-\frac{\es}{2}$}}   to (2.5,3.75)  ;
			\draw (3.5,0)node [anchor=north] {{\scriptsize $+\frac{\es}{2}$}}to (3.5,3.75); 
			\draw (2.5,.75) to (3.5,.75);
			\draw(2.5,1.5) to (3.5,1.5);
			\draw(2.5,2.25) to (3.5,2.25);
			\draw(2.5,3) to (3.5,3);
			\draw [fill](2.75,.25) to (3.25,.25) to(3.25,.50) to (2.75,.50) to (2.75,.25);
			\draw [fill](2.75,1) to (3.25,1) to(3.25,1.25) to (2.75,1.25) to (2.75,1);
			\draw [fill](2.75,1.75) to (3.25,1.75) to(3.25,2) to (2.75,2) to (2.75,1.75);
			\draw [fill](2.75,2.5) to (3.25,2.5) to(3.25,2.75) to (2.75,2.75) to (2.75,2.5);
			\draw [fill](2.75,3.25) to (3.25,3.25) to(3.25,3.5) to (2.75,3.5) to (2.75,3.25);
			\draw [->,photon](0,3.5)  to (2.5,3.5) ;
			\draw [->,photon](0,3.0)  to (2.5,3.0) ;
			\draw [->,photon](0,2.5)  to (2.5,2.5) ;
			\draw [->,photon](0,2)  to (2.5,2) ;
			\draw [->,photon](0,1.5)  to (2.5,1.5) ;
			\draw [->,photon](0,1)  to (2.5,1) ;
			\draw [->,photon](0,0.5)  to (2.5,0.5) ;
			\draw [->,photon](3.5,3.25)  to (6,3.25) ;
			\draw [->,photon](3.5,2.5)  to (6,2.5) ;
			\draw [->,photon](3.5,1.75)  to (6,1.75) ;
			\draw [->,photon](3.5,1)  to (6,1) ;
			\draw [->,photon](3.5,0.25)  to (6,0.25) ;
			\filldraw[black] (2,4) circle (0pt) node[anchor=east]{\Large $\Omega_L^\es$};
			\filldraw[black] (3.5,4) circle (0pt) node[anchor=east]{\Large $\Omega_M^\es$};
			\filldraw[black] (5,4) circle (0pt) node[anchor=east]{\Large $\Omega_R^\es$};
 			\draw [<->] (2.4,.75) to (2.4,1.5);
			\draw (2.4,1.2) node[anchor=east] {{\scriptsize $\varepsilon$}};
			\end{tikzpicture}
			\caption{Schematic representation of a microscopic model. }
			\label{micmod}
		\end{center}
	\end{figure}
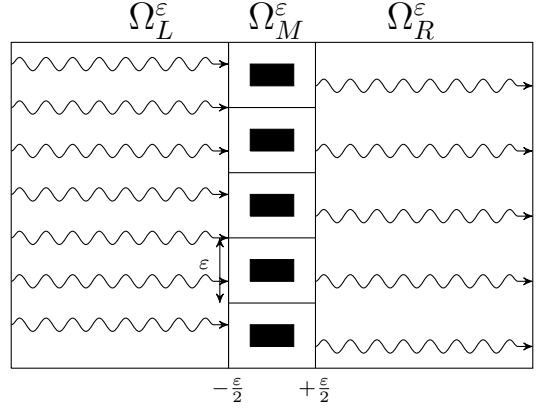
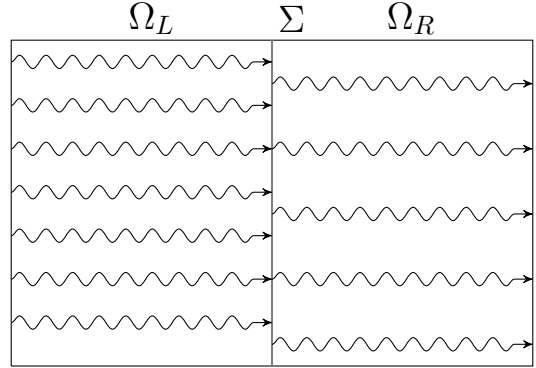
\begin{figure}[H]
		\begin{center}
			\begin{tikzpicture}[scale=1.15,
  >=stealth',
  pos=.8,
  photon/.style={decorate,decoration={snake,post length=1mm}}
]\draw (0,0) node [anchor=north] {} to (6,0) node [anchor=north] {}  to (6,3.75) to (0,3.75) to (0,0);
		
			\draw [gray,thick](3,0) to (3,3.75)  ;
			\draw [->,photon](0,3.5)  to (3,3.5) ;
			\draw [->,photon](0,3.0)  to (3,3.0) ;
			\draw [->,photon](0,2.5)  to (3,2.5) ;
			\draw [->,photon](0,2)  to (3,2) ;
			\draw [->,photon](0,1.5)  to (3,1.5) ;
			\draw [->,photon](0,1)  to (3,1) ;
			\draw [->,photon](0,0.5)  to (3,0.5) ;
			\draw [->,photon](3.,3.25)  to (6,3.25) ;
			\draw [->,photon](3.,2.5)  to (6,2.5) ;
			\draw [->,photon](3.,1.75)  to (6,1.75) ;
			\draw [->,photon](3.,1)  to (6,1) ;
			\draw [->,photon](3.,0.25)  to (6,0.25) ;
			\filldraw[black] (2,4) circle (0pt) node[anchor=east]{\Large $\Omega_L$};
			\filldraw[black] (3.5,4) circle (0pt) node[anchor=east]{\Large $\Sigma$};
			\filldraw[black] (5,4) circle (0pt) node[anchor=east]{\Large $\Omega_R$};
			\end{tikzpicture}
			\caption{Schematic representation of the macroscopic model. }
			\label{macmod}
		\end{center}
	\end{figure}

In the domain $\Omega$, we consider the following system of wave equations for the unknown triplet $(u_L^\es,u_M^\es,u_R^\es):(0,T)\times\Omega\rightarrow \mathbb{R}$,
\begin{subequations}
\begin{align}
 \dfrac{\partial}{\partial t}\left( {\rho_L}\dfrac{\partial u_L^\es}{\partial t}\right)-\mathrm{div}(D_L \nabla u_L^\es)&=f_L\qquad &\mbox{in}\qquad (0,T)\times \oll\label{eq:microscopicul},\\
    \dfrac{1}{\es} \dfrac{\partial}{\partial t}\left({\rho_M^\es} \dfrac{\partial u_M^\es}{\partial t}\right)-\dfrac{1}{\es}\mathrm{div}\left(D_M^\es \nabla u_M^\es\right)&=0\qquad &\mbox{in}\qquad (0,T)\times\omm\label{eq:microscopicum},\\
  \dfrac{\partial}{\partial t}\left({\rho_R} \dfrac{\partial u_R^\es}{\partial t}\right)-\mathrm{div}(D_R \nabla u_R^\es)&=0\qquad &\mbox{in}\qquad (0,T)\times\orr\label{eq:microscopicur}.
  \end{align}

We endow the system \eqref{eq:microscopicul}--\eqref{eq:microscopicur} with the perfect transmission conditions
\begin{align}
    u_L^\es&=u_M^\es \qquad \mbox{on}\qquad  (0,T)\times L^\es,\\
    u_R^\es&=u_M^\es \qquad \mbox{on}\qquad  (0,T)\times R^\es\label{eq:jumpconditionur},\\
    D_L\nabla u_L^\es\cdot \nu&=\dfrac{1}{\es}D_M^\es \nabla u_M^\es\cdot \nu\qquad \mbox{on} \qquad  (0,T)\times L^\es,\\
     D_R\nabla u_R^\es\cdot \nu&=\dfrac{1}{\es}D_M^\es \nabla u_M^\es\cdot \nu \qquad \mbox{on}\qquad (0,T)\times R^\es,\label{eq:jumpconditionforflux}
\end{align}
\end{subequations}
where $\nu$ is the  unit normal vector pointing from the left bulk region to the right.
On the outer boundary $\partial\Omega$, we assume periodic boundary conditions (in other words, we can identify $\Omega$ with a two-dimensional torus $\mathbb{T}$).
For the initial condition we consider $u_L^\es(0)=v_1,u_M^\es(0)=u_R^\es(0)=0$ and $\dfrac{\partial u_L^\es}{\partial t}(0)=v_2,\dfrac{\partial u_M^\es}{\partial t}(0)=\dfrac{\partial u_R^\es}{\partial t}(0)=0$ in their respective domains.  Note that the source term as well as initial conditions have support in $\oll$, this is to ensure that the wave is excited from the left bulk domain and transmitted through the heterogeneous layer. An analogous configuration could be considered by placing the source term and the initial data in the right bulk domain.

Below, we list some of the notation and functional spaces that we use throughout the paper. Additionally, standard notation on Lebesgue, Sobolve and Bochner spaces is employed throughout.

\subsection{Notations and functional spaces}
For simplicity, we set $\T,Z:=(0,1)$, further more we have the following definitions
\begin{itemize}
    \item We define $\bar x:=(0,x_2)$ and $\nabla_{\bar x}:=\left(0,\frac{\p}{\p x_2}\right)^{\operatorname{T}}.$
    \item  We denote $\Gamma_L$ and $\Gamma_R$ as left and right vertical boundaries of $\Omega.$
    \item  $\hat{\rho}^\es:=\begin{cases}
        \rho_L \quad\mbox{in}\quad [0,T]\times\oll\\
       \frac{1}{\es} \rho_M^\es \quad\mbox{in}\quad [0,T]\times\omm\\
        \rho_R \quad\mbox{in}\quad [0,T]\times\orr
    \end{cases}\qquad
  \hat{D}^\es:=\begin{cases}
        D_L \quad\mbox{in}\quad [0,T]\times\oll\\
        \frac{1}{\es} D_M^\es \quad\mbox{in}\quad [0,T]\times\omm\\
        D_R \quad\mbox{in}\quad [0,T]\times\orr
    \end{cases}$
    \item $\{e_1,e_2\}$ are the standard basis for $\mathbb{R}^2$.
    \item $H_{\sharp,0}^1(Z):=\{v\in {H}_{\sharp}^1(Z):\int_Zv \di{z}=0\}.$
    \item  $\mathcal{H}_{\sharp}(\T,Y):=\{v\in H^1(\T\times Y):v \mbox{ is periodic in } \tau \mbox{ and } y_2\}. $
    \item  $\mathcal{H}_{\sharp,0}(\T,Y):=\{v\in \mathcal{H}_{\sharp}(\T,Y):\int_\T\int_Yv \di{y}\di{\tau}=0\}. $
    \item  $\mathcal{C}_{\sharp}(\overline{\T\times Y\times Z}):=\{v(\tau,y,z)\in C(\overline{\T\times Y\times Z}):v \mbox{ is periodic in } \tau,y_2 \mbox{ and } z\}.$
    \item $\mathcal{H}_{\sharp}(\Omega_i):\{v\in H^1(\Omega_i): v \mbox{ is periodic in } x_2 \mbox{ direction}\},$ where $i\in\{L,M,R\}$.
    \item $H^n_{\text{pw}}(\Omega):=\{v\in L^2(\Omega):v|_{\Omega_i^\es}\in H^n(\Omega_i^\es) \text{  for  } i=L,M,R\}$


\end{itemize}
Now, we list the assumptions on the data that we require to establish the result in the paper.
\subsection{Assumptions on the data}
\begin{itemize}
    \item [(A1)] The coefficients $D_L,D_R,\rho_L,\rho_R>0$ are constants and we consider the time- and space-dependent coefficients $\rho_M^\es(t,x):=\rho_M\left(t,x_2,\frac{t}{\es},\frac{x}{\es},\frac{x_2}{\es^2}\right)$, 
$D_M^\es(t,x):= D_M\left(t,x_2,\frac{t}{\es},\frac{x}{\es},\frac{x_2}{\es^2}\right)$, where $ \rho_M,D_M  \in W^{1,\infty}((0,T)\times \Sigma\times \T\times Y\times Z)$ are periodic in $\tau, y_2$ and $z$ variable. In addition we assume there exist positive constants $ \rho_{min}$, $\rho_{max}$, $D_{min}$ and $ D_{max}$ such that $\rho_{min}\leq\rho_M^\es(t,x)\leq\rho_{max}$ and  $D_{min}\leq D_M^\es(t,x)\leq D_{max}$.

\item[(A2)] For the initial conditions and source term, we assume $v_1\in C_c^2(\mol), v_2\in C_c^1(\mol)$ and $f_L\in C_c^2((0,T)\times \mol)$.
 Moreover, we assume that $\es$ sufficiently small so that $v_1\in C_c^2(\oll), v_2\in C_c^1(\oll)$ and $f_L\in C_c^2((0,T)\times \oll)$ for every $\es$.
\item[(A3)] For fixed $(t,x_2)\in (0,T)\times \Sigma$,
let $\mathcal{L}:\mathcal{H}_{\sharp}(\T,Y)\rightarrow \mathcal{H}_{\sharp}(\T,Y)^*$ denote the  wave operator 
 \begin{align*}
         \mathcal{L}(v):= \frac{\partial}{\partial \tau}\left(\bar \rho_M \frac{\partial v}{\partial \tau}\right)-\mathrm{div}_y\left(\bar D_M\nabla_yv\right)
     \end{align*}
     with the periodic boundary condition with respect to the variables $\tau$ and $y_2$ and the homogeneous Neumann boundary condition
     \begin{align*}
         \left(\bar D_M\nabla_yv\right)\cdot \nu=0 \qquad\mbox{ on }\qquad Y_L\cup Y_R
     \end{align*}
     where $\bar \rho_M:=\int_Z\rho_M\di{z}$ and $\bar D_M=\begin{bmatrix}
            \int_Z D_M\di{z}&0\\
            0&  \left(\int_Z\frac{1}{ D_M}\di{z}\right)^{-1}\end{bmatrix}$. 
            Then we assume the range of $\mathcal{L}$ is closed in $\mathcal{H}_{\sharp}(\T,Y)^*$ and $\operatorname{Ker}(\mathcal{L})=\operatorname{span}\{1\}.$

\end{itemize}

\begin{remark}
In our microscopic model, we consider the coefficients depending on two temporal scales and three spatial scales. The variable $\frac{x_2}{\es}$ describes oscillations at scale $\es$, and the variable $\frac{x_2}{\es^2}$ represents a finer and more rapidly oscillating structure in the vertical direction. The additional scale can be interpreted as stratification or finer substructure of the layer, naturally leading to a reiterated homogenization analysis. Without much difficulty, the methodology and the analysis of our paper can extend to an arbitrary number of well-separated spatial and temporal scales.
\end{remark}
\begin{remark}
    The Assumptions $(A1)$ and $(A2)$ are natural, and we use them throughout the paper. The Assumption $(A3)$ means a nonresonance condition at the microscopic scale. This assumption is only  used to write our macroscopic model in terms of suitable cell problems. Specifically, it is needed for the uniqueness of the hyperbolic cell problem. However, if the coefficients $\rho_M^\es$ and $D_M^\es$ are independent of the fast time scale $\frac{t}{\es}$, then  the assumption $(A3)$ is no longer needed. Furthermore, if the coefficients are independent of the higher-order spatial scale $\frac{x_2}{\es^2}$, then we get $\bar\rho_M=\rho_M$ and $\bar D_M:=D_MI$.
\end{remark}
\subsection{Weak formulation and solvability}

Here we introduce the concept of a weak formulation to the microscopic problem and discuss the solvability of the problem.
\begin{definition}
 The weak formulation of the problem \eqref{eq:microscopicul}--\eqref{eq:jumpconditionforflux}is to 
find $u^\es\in L^2(0,T;H_{\sharp}^1(\Omega))$ with $\frac{\partial u^\es}{\p t}\in L^2(0,T;L^2(\Omega))$ such that 
\begin{multline}\label{eq:weakform}
    -\rho_L\int_0^T\int_{\oll}  \dfrac{\partial u_L^\es}{\partial t}  \dfrac{\partial \phi_L}{\partial t}\di{x}\di{t}-  \dfrac{1}{\es} \int_0^T\int_{\omm}  \left({\rho_M^\es} \dfrac{\partial u_M^\es}{\partial t}\right) \dfrac{\partial \phi_M}{\partial t}\di{x}\di{t} -    \rho_R\int_0^T\int_{\orr} \dfrac{\partial u_R^\es}{\partial t}\dfrac{\partial \phi_R}{\partial t} \di{x}\di{t}\\
    +D_L\int_0^T\int_{\oll}\nabla u_L^\es \cdot \nabla \phi_L   \di{x}\di{t} +   \dfrac{1}{\es}\int_0^T\int_{\omm}  D_M^\es  \nabla u_M^\es\cdot\nabla \phi_M \di{x}\di{t}+D_R\int_0^T\int_{\orr} \nabla u_R^\es\cdot\nabla \phi_R \di{x} \di{t}
   \\
    = \int_0^T\int_{\oll} f_L\phi_L \di{x}\di{t}+\int_{\oll}\rho_L v_2\phi_L(0)\di{x}
\end{multline}
for all test function $\phi\in L^2((0,T;H^1_{\sharp}(\Omega))$ with $\frac{\p \phi}{\p t}\in L^2((0,T;L^2(\Omega))$ and $\phi(T)=0$ with $u^\es(0)=v_1$.
\end{definition}

Equivalently, the identity \eqref{eq:weakform} can be written as
\begin{multline*}
    -\int_0^T\int_\Omega \hat{\rho^\es}\dfrac{\partial u^\es}{\partial t}  \dfrac{\partial \phi}{\partial t}\di{x}\di{t}+\int_0^T\int_\Omega\hat{D^\es}\nabla u^\es \cdot \nabla\phi \di{x}\di{t}=\int_0^T\int_\Omega f_L\phi\di{x}\di{t}+\int_{\Omega}\rho_L v_2\phi(0)\di{x}.
\end{multline*}

 The aforementioned problem is a linear wave equation with space-time-dependent coefficients. From assumptions $(A1)$ and $(A2)$, for any fixed $\es$,  the existence and uniqueness of the weak solution follow from energy estimates and the Galerkin method (we refer to \cite{casado2011homogenization} for details). Furthermore, from the standard regularity lifting argument, we get the following additional regularity for the solution
 \begin{align*}
     \frac{\p^2 u^\es }{\p t^2}\in L^2(0,T;L^2(\Omega)) \quad \frac{\p u^\es }{\p t}\in L^2(0,T;H^1_{\text{pw}}(\Omega))\quad u^\es \in L^2(0,T;H^2_{\text{pw}}(\Omega)).
 \end{align*}

 \section{Uniform energy estimates and compactness results}\label{section:Uniform energy estimates and compactness results}

 In this section, we discuss the uniform energy estimates and related compactness results for the microscopic problem. In the first part of the section, we examine the conditions under which the uniform energy bounds can be obtained.  We then study the microscopic solution in the bulk domains. Using the energy estimates, we show that as  $\es\rightarrow 0$, the solution converges strongly in $L^2$ while its first-order spatial and temporal derivatives converge weakly in $L^2$.
 Next, using the energy estimate and the trace theorem, we derive a compactness result on the interface $L^\es$ and $R^\es$.  

 To find the limiting behavior of the solution in the heterogeneous layer, we require a tool that simultaneously accounts for homogenization with multiple  scales and dimension reduction. Therefore, we introduce the concept of multiscale convergence for thin layer and derive associated compactness results.   Finally, we determine the homogenization limit of $u_M^\es, \nabla u_M^\es$ and $\frac{\p  u_M^\es}{\p t}$ in the layer.

\subsection{Energy estimates}
To find the effective model for the microscopic wave propagation problem, we require a uniform energy bound of the form

 \begin{subequations}
	\begin{align}
     \left\|\ u_L^\es\right\|_{L^{\infty}(0,T;L^2(\oll))}+  \left\|\ u_R^\es\right\|_{L^{\infty}(0,T;L^2(\orr))}+  \dfrac{1}{\sqrt{\es}}\left\|\ u_M^\es\right\|_{L^{\infty}(0,T;L^2(\omm))}&\leq C,\label{eq:energy_estimate_1_case_1}\\
	    \left\|\ \dfrac{\partial u_L^\es}{\partial t}\right\|_{L^{\infty}(0,T;L^2(\oll))}+  \left\|\ \dfrac{\partial u_R^\es}{\partial t}\right\|_{L^{\infty}(0,T;L^2(\orr))}+  \dfrac{1}{\sqrt{\es}}\left\|\ \dfrac{\partial u_M^\es}{\partial t}\right\|_{L^{\infty}(0,T;L^2(\omm))}&\leq C,\label{eq:energy_estimate_2_case_1}\\
	    \|\nabla u_L^\es\|_{L^{\infty}(0,T;L^2(\oll))}+\|\nabla u_R^\es\|_{L^{\infty}(0,T;L^2(\orr))}+\dfrac{1}{\sqrt{\es}}\|\nabla u_M^\es\|_{L^{\infty}(0,T;L^2(\omm))}&\leq C\label{eq:energy_estimate_3_case_1},
	\end{align}
\end{subequations}
 where $C>0$ is a constant independent of $\es$.
 Since we have a system of linear wave equations with smooth data, one should expect the solution to have a uniform energy bound. But, in general, this is not true due to the lack of energy conservation caused by the time-dependent coefficients. Here we discuss two physically relevant choices of $\rho_M$ and $D_M$ for which  we show that there exists a uniform energy bound, but with an example, we discuss the possibilities of the non-stability of the solution.
 \subsubsection{Case 1: no oscillation on the fast temporal scale}
 In this case, we assume that the data $\rho^\es$ and $D^\es$ do not oscillate with respect to the fast time variable. In other words, we assume $\rho_M\left(t,x_2,\frac{t}{\es},\frac{x}{\es},\frac{x_2}{\es^2}\right)\equiv \rho_M\left(t,x_2,\frac{x}{\es},\frac{x_2}{\es^2}\right)$, $D_M\left(t,x_2,\frac{t}{\es},\frac{x}{\es},\frac{x_2}{\es^2}\right)\equiv D_M\left(t,x_2,\frac{x}{\es},\frac{x_2}{\es^2}\right)$.
 We denote this assumption by $(B1)$. Structures of this type are commonly encountered in the modeling of metasurfaces and metascreens.  We now show that under these assumptions we obtain the required energy bound.
	\begin{lemma}\label{lemma:energy_estimate_for_case_1}
	Assume  $(A1),(A2)$ and $(B1)$ holds. Then the weak solution $u^\es$ to the microscopic problem \eqref{eq:microscopicul}--\eqref{eq:jumpconditionforflux} satisfies the uniform energy estimates
    \eqref{eq:energy_estimate_1_case_1}--\eqref{eq:energy_estimate_3_case_1}.
	\end{lemma}
	\begin{proof}
	Let $t\in(0,T)$, we derive the energy estimate by choosing the test function $\phi= \mathds{1}_{(0,t)}\dfrac{\partial u^\es}{\partial t}$ in \eqref{eq:weakform}. We get
	\begin{multline*}
    \rho_L\int_0^t\int_{\oll}  \dfrac{\partial^2u_L^\es}{\partial t^2}  \dfrac{\partial u_L^\es}{\partial t}\di{x}\di{t}+  \dfrac{1}{\es}\int_0^t\int_{\omm} \dfrac{\partial }{\partial t}\left(\rho_M^\es \dfrac{\partial u_M^\es}{\partial t}\right)  \dfrac{\partial u_M^\es}{\partial t}\di{x}\di{t}+\rho_R \int_0^t\int_{\orr}  \dfrac{\partial^2u_R^\es}{\partial t^2}  \dfrac{\partial u_R^\es}{\partial t}\di{x}\di{t}\\
    +D_L\int_0^t\int_{\oll}\nabla u_L^\es\cdot \nabla  \left(\dfrac{\partial u_L^\es}{\partial t} \right) \di{x} \di{t}
    + D_R\int_0^t\int_{\orr} \nabla u_R^\es\cdot\nabla \left(\dfrac{\partial u_R^\es}{\partial t} \right) \di{x} \di{t}
   \\ + \dfrac{1}{\es}\int_0^t\int_{\omm}  D_M^\es  \nabla u_M^\es\cdot\nabla \left(\dfrac{\partial u_M^\es}{\partial t} \right) \di{x}\di{t} = \int_0^t\int_{\oll} f_L\left(\dfrac{\partial u_L^\es}{\partial t} \right) \di{x}\di{t}.
\end{multline*}

By rearranging the terms, we obtain
	\begin{multline*}
   \dfrac{\rho_L}{2}\int_0^t\dfrac{d}{dt}\left\|\dfrac{\partial u_L^\es}{\partial t} \right\|_{L^2(\oll)}^2\di{t}+   \dfrac{1}{2\es}\int_0^t\dfrac{d}{dt}\left\|\sqrt{\rho_M^\es}\dfrac{\partial u_M^\es}{\partial t} \right\|_{L^2(\omm)}^2\di{t}+  \dfrac{\rho_R}{2}\int_0^t\dfrac{d}{dt}\left\|\dfrac{\partial u_R^\es}{\partial t} \right\|_{L^2(\orr)}^2\di{t}\\
   + \dfrac{D_L}{2}\int_0^t\dfrac{d}{dt}\int_{\oll}|\nabla u_L^\es|^2\di{x}\di{t}+
   \dfrac{D_R}{2}\int_0^t\dfrac{d}{dt}\int_{\orr}|\nabla u_R^\es|^2\di{x}\di{t}
   +\dfrac{1}{2\es}\int_0^t\dfrac{d}{dt}\int_{\omm}D_M^\es |\nabla u_M^\es|^2 \di{x}\di{t}\\
     = \int_0^t\int_{\oll} f_L\left(\dfrac{\partial u_L^\es}{\partial t} \right) \di{x}\di{t}+\dfrac{1}{2\es}\int_0^t\int_{\omm}\dfrac{\partial \rho_M^{\es}}{\partial t}\left(\dfrac{\partial u_M^\es}{\partial t}\right)^2\di{x}\di{t}+\dfrac{1}{2\es}\int_0^t\int_{\omm}\dfrac{\partial D_M^\es}{\partial t} |\nabla u_M^\es|^2 \di{x}\di{t}.
\end{multline*}
Fundamental theorem of calculus yields
	\begin{multline*}
   \dfrac{\rho_L}{2}\left\|\dfrac{\partial u_L^\es}{\partial t} \right\|_{L^2(\oll)}^2+   \dfrac{1}{2\es}\left\|\sqrt{\rho_M^\es}\dfrac{\partial u_M^\es}{\partial t} \right\|_{L^2(\omm)}^2+  \dfrac{\rho_R}{2}\left\|\dfrac{\partial u_R^\es}{\partial t} \right\|_{L^2(\orr)}^2\\
   + \dfrac{D_L}{2}\int_{\oll}|\nabla u_L^\es|^2\di{x}+
   \dfrac{D_R}{2}\int_{\orr}|\nabla u_R^\es|^2\di{x}
   +\dfrac{1}{2\es}\int_{\omm}D_M^\es |\nabla u_M^\es|^2 \di{x}\\
     = \int_0^t\int_{\oll} f_L\left(\dfrac{\partial u_L^\es}{\partial t} \right) \di{x}\di{t}+\dfrac{1}{2\es}\int_0^t\int_{\omm}\dfrac{\partial \rho_M^{\es}}{\partial t}\left(\dfrac{\partial u_M^\es}{\partial t}\right)^2\di{x}\di{t}+\dfrac{1}{2\es}\int_0^t\int_{\omm}\dfrac{\partial D_M^\es}{\partial t} |\nabla u_M^\es|^2 \di{x}\di{t}\\
     +\dfrac{\rho_L}{2}\|v_2\|_{L^2(\oll)}^2+\dfrac{D_L}{2}\|\nabla v_1\|_{L^2(\oll)}^2.
\end{multline*}
Now using the assumptions $(A1)$ and $(B1)$ together with Cauchy-Schwarz inequality, we obtain
	\begin{multline*}
   \dfrac{\rho_L}{2}\left\|\dfrac{\partial u_L^\es}{\partial t} \right\|_{L^2(\oll)}^2+   \dfrac{\rho_{min}}{2\es}\left\|\dfrac{\partial u_M^\es}{\partial t} \right\|_{L^2(\omm)}^2+  \dfrac{\rho_R}{2}\left\|\dfrac{\partial u_R^\es}{\partial t} \right\|_{L^2(\orr)}^2\\
   + \dfrac{D_L}{2}\int_{\oll}|\nabla u_L^\es|^2\di{x}+
   \dfrac{D_R}{2}\int_{\orr}|\nabla u_R^\es|^2\di{x}
   +\dfrac{D_{min}}{2\es}\int_{\omm} |\nabla u_M^\es|^2 \di{x}\\
     \leq \|f_L\|_{L^2(0,t;L^2(\oll))}^2+\left\|\dfrac{\partial u_L^\es}{\partial t}\right\|_{L^2(0,t;L^2(\oll))}^2+\dfrac{C}{\es}\left\|\partial_tu_M^\es\right\|_{L^2(0,t;L^2(\omm))}^2+\dfrac{C}{\es}\left\|\nabla u_M^\es\right\|_{L^2(0,t;L^2(\omm))}^2\\
     +\dfrac{\rho_L}{2}\|v_2\|_{L^2(\oll)}^2+\dfrac{D_L}{2}\|v_1\|_{H^1(\oll)}^2.
\end{multline*}
From Grönwall's inequality and $(A2)$, we get the estimates \eqref{eq:energy_estimate_2_case_1} and \eqref{eq:energy_estimate_3_case_1}. It remains to show the estimate \eqref{eq:energy_estimate_1_case_1}. 
 From the fundamental theorem of calculus, for $i\in\{L,M,R\}$ we have
    \begin{align*}
        u_i^\es(t)=u_i^\es(0)+\int_0^t \dfrac{\partial u_i^\es(s)}{\partial t}\di{s}.
    \end{align*}
    Taking $L^2$ norm on both sides, using Minkowski's inequality together with $(A2)$ and \eqref{eq:energy_estimate_3_case_1}, we get \eqref{eq:energy_estimate_1_case_1}.
	\end{proof}
\subsubsection{Case 2: Moving heterogeneous layer }
In Case 1, we avoided the fast temporal oscillation to derive the energy bound. Here we give a physically relevant example, where we have both spatial and temporal oscillations. In this particular scenario, at the microscale, we assume the material property of the periodic layer moves vertically in a wave-like fashion with a velocity $c$. In other words, we assume $\rho_M\left(t,x_2,\frac{t}{\es},\frac{x}{\es},\frac{x_2}{\es^2}\right)\equiv \rho_M\left(\frac{x_1}{\es},\frac{x_2-ct}{\es}\right)$ and $D_M\left(t,x_2,\frac{t}{\es},\frac{x}{\es},\frac{x_2}{\es^2}\right)\equiv D_M\left(\frac{x_1}{\es},\frac{x_2-ct}{\es}\right)$, where $c$ is a positive constant. In addition, we assume  there exists $\lambda_L,\lambda_M,\lambda_R>0$ such that $D_{L}-c^2\rho_{L}>\lambda_L$, $D_{R}-c^2\rho_{R}>\lambda_R$ and $D_M^\es-\rho_M^\es c^2>\lambda_M$.
We denote these assumptions as $(B2)$.  Such models are useful to construct time-modulated laminated media relevant to phononic metamaterials and metasurfaces designed for wave filtering, frequency conversion, directional band gaps, asymmetric transmission, and non-reciprocal elastic or acoustic wave propagation \cite{Nassar2017Modulated}. We show that under this space-time modulated structure, we get the uniform energy bound.
\begin{lemma}
	   Assume $(A1),(A2)$ and $(B2)$ holds. Then the weak solution $u^\es$ to the microscopic problem \eqref{eq:microscopicul}--\eqref{eq:jumpconditionforflux} satisfies the energy estimates \eqref{eq:energy_estimate_1_case_1}--\eqref{eq:energy_estimate_3_case_1}.
	\end{lemma}
	\begin{proof}
	We first derive the energy estimate for our problem in a moving coordinate system. For that, we use the following change of variables 
	\begin{align*}
	    X_1&:=x_1 & X_2&:=x_2-ct, 
	\end{align*}
	and we denote
    \begin{align*}
        \widetilde{u}_{L}^\es(t,X):={u}_{L}^\es(t,x)\qquad \widetilde{u}_{R}^\es(t,X):={u}_{R}^\es(t,x) \qquad \widetilde{u}_{M}^\es(t,X):={u}_{M}^\es(t,x),
    \end{align*}
    \begin{align*}
        \widetilde{f}_L(t,X):=f_L(t,x) \qquad \widetilde{\phi}(t,X):=\phi(t,x).
    \end{align*}
    From the change of variables, we have 
    \begin{align*}
        \rho_M^\es(t,x)=\rho_M(\tfrac{X}{\es})\qquad  D_M^\es(t,x)=D_M(\tfrac{X}{\es}),
    \end{align*}    
  and  the following chain rules
	\begin{align*}
	    \dfrac{\partial u_{i}^\es(t,x)}{\partial t}&=\dfrac{\partial \widetilde{u}_{i}^\es(t,X)}{\partial t}-c\dfrac{\partial \widetilde{u}_{i}^\es(t,X)}{\partial{X_2}},\\
        \dfrac{\partial \phi_{i}(t,x)}{\partial t}&=\dfrac{\partial \widetilde{\phi}_{i}(t,X)}{\partial t}-c\dfrac{\partial \widetilde{\phi}_{i}(t,X)}{\partial{X_2}},\\ 
        \dfrac{\partial \widetilde{u}^\es}{\partial t}(0)&=v_2(X)+c\frac{\p v_1(X)}{\p X_2}\\
	    \nabla u_{j}(t,x)&=\nabla_X\widetilde{u}_{j}(t,X),
	\end{align*}
 where $i\in\{L,M,R\}$ and $j\in\{L,R\}$.  
Moreover, using the change of variable in the weak formulation \eqref{eq:weakform}, we get
\begin{multline*}
    -\rho_L\int_0^T\int_{\oll} \left(\dfrac{\partial \widetilde{u}_{L}^\es(t,X)}{\partial t}-c\dfrac{\partial \widetilde{u}_{L}^\es(t,X)}{\partial{X_2}}\right)\left(\dfrac{\partial \widetilde{\phi}_{L}(t,X)}{\partial t}-c\dfrac{\partial \widetilde{\phi}_{L}(t,X)}{\partial{X_2}}\right)\di{X\di{t}}\\
     -\frac{1}{\es}\int_0^T\int_{\omm} \rho_M^\es \left(\dfrac{\partial \widetilde{u}_{M}^\es(t,X)}{\partial t}-c\dfrac{\partial \widetilde{u}_{M}^\es(t,X)}{\partial{X_2}}\right)\left(\dfrac{\partial \widetilde{\phi}_{M}(t,X)}{\partial t}-c\dfrac{\partial \widetilde{\phi}_{M}(t,X)}{\partial{X_2}}\right)\di{X\di{t}}\\
      -\rho_R\int_0^T\int_{\orr} \left(\dfrac{\partial \widetilde{u}_{R}^\es(t,X)}{\partial t}-c\dfrac{\partial \widetilde{u}_{R}^\es(t,X)}{\partial{X_2}}\right)\left(\dfrac{\partial \widetilde{\phi}_{R}(t,X)}{\partial t}-c\dfrac{\partial \widetilde{\phi}_{R}(t,X)}{\partial{X_2}}\right)\di{X\di{t}}\\
       +D_L\int_0^T\int_{\oll}\nabla_X \widetilde{u}_L^\es\cdot\nabla_X  \widetilde{\phi}_L \di{X}\di{t}+\dfrac{1}{\es}\int_0^T\int_{\omm}\left({D}_M^\es \nabla_X \widetilde{u}_M^\es\right)\cdot\nabla_X  \widetilde{\phi}_M\di{X}\di{t}\\
        +D_R\int_0^T\int_{\orr} \nabla_X \widetilde{u}_R^\es\cdot\nabla_X  \widetilde{\phi}_R \di{X}\di{t}=\int_0^T\int_{\oll}\widetilde{f}_L \widetilde{\phi}_L\di{X}\di{t}+\rho_L\int_{\oll}v_2(X)\widetilde{\phi}_L(0)\di{X}
\end{multline*}
By rearranging the terms, we have

\begin{multline*}
    -\rho_L\int_0^T\int_{\oll}\left[ \left(\dfrac{\partial \widetilde{u}_{L}^\es(t,X)}{\partial t}-c\dfrac{\partial \widetilde{u}_{L}^\es(t,X)}{\partial{X_2}}\right)\dfrac{\partial \widetilde{\phi}_{L}(t,X)}{\partial t}-c\dfrac{\partial \widetilde{u}_{L}^\es(t,X)}{\partial{t}}\dfrac{\partial \widetilde{\phi}_{L}(t,X)}{\partial X_2}\right]\di{X}\di{t}\\
     -\frac{1}{\es}\int_0^T\int_{\omm} \rho_M^\es\left[ \left(\dfrac{\partial \widetilde{u}_{M}^\es(t,X)}{\partial t}-c\dfrac{\partial \widetilde{u}_{M}^\es(t,X)}{\partial{X_2}}\right)\dfrac{\partial \widetilde{\phi}_{M}(t,X)}{\partial t}-c\dfrac{\partial \widetilde{u}_{M}^\es(t,X)}{\partial{t}}\dfrac{\partial \widetilde{\phi}_{M}(t,X)}{\partial X_2}\right]\di{X}\di{t}\\
      -\rho_R\int_0^T\int_{\orr} \left[\left(\dfrac{\partial \widetilde{u}_{R}^\es(t,X)}{\partial t}-c\dfrac{\partial \widetilde{u}_{R}^\es(t,X)}{\partial{X_2}}\right)\dfrac{\partial \widetilde{\phi}_{R}(t,X)}{\partial t}-c\dfrac{\partial \widetilde{u}_{R}^\es(t,X)}{\partial{t}}\dfrac{\partial \widetilde{\phi}_{R}(t,X)}{\partial X_2}\right]\di{X\di{t}}\\
       +\int_0^T\int_{\oll}\widetilde{D}_L\nabla_X \widetilde{u}_L^\es\cdot\nabla_X  \widetilde{\phi}_L \di{X}\di{t}+\dfrac{1}{\es}\int_0^T\int_{\omm}\left(\widetilde{D}_M^\es\nabla_X \widetilde{u}_M^\es\right)\cdot\nabla_X  \widetilde{\phi}_M\di{X}\di{t}\\
        +\int_0^T\int_{\orr} \widetilde{D}_R\nabla_X \widetilde{u}_R^\es\cdot\nabla_X  \widetilde{\phi}_R \di{X}\di{t}=\int_0^T\int_{\oll}\widetilde{f}_L \widetilde{\phi}_L\di{X}\di{t}+\rho_L\int_{\oll}v_2(X) \widetilde{\phi}_L(0)\di{X},
\end{multline*}
where $\widetilde{D}_L:=\begin{bmatrix}
D_L & 0 \\
0 & D_L-\rho_Lc^2
\end{bmatrix}$, $\widetilde{D}_M^\es:=\begin{bmatrix}
D_M^\es & 0 \\
0 & D_M^\es-\rho_M^\es c^2
\end{bmatrix}$ and $\widetilde{D}_R:=\begin{bmatrix}
D_R & 0 \\
0 & D_R-\rho_Rc^2
\end{bmatrix}$. 

By performing integration by parts with respect to the time variable as well as the $X_2$ variable, and using $\widetilde{\phi}(T)=0$, we arrive at

\begin{multline*}
    \rho_L\int_0^T\int_{\oll} \frac{\p}{\p t}\left(\dfrac{\partial \widetilde{u}_{L}^\es(t,X)}{\partial t}-2c\dfrac{\partial \widetilde{u}_{L}^\es(t,X)}{\partial{X_2}}\right) \widetilde{\phi}_{L}(t,X)\di{X}\di{t}\\
     +\frac{1}{\es}\int_0^T\int_{\omm}\rho_M^\es \frac{\p}{\p t}\left(\dfrac{\partial \widetilde{u}_{M}^\es(t,X)}{\partial t}-c\dfrac{\partial \widetilde{u}_{M}^\es(t,X)}{\partial{X_2}}\right)\widetilde{\phi}_{M}(t,X)\di{X}\di{t}\\
     -\frac{c}{\es}\int_0^T\int_{\omm} \frac{\p}{\p X_2}\left(\rho_M^\es \dfrac{\partial \widetilde{u}_{M}^\es(t,X)}{\partial{t}}\right)\widetilde{\phi}_{M}(t,X)\di{X}\di{t}\\
      +\rho_R\int_0^T\int_{\orr}\frac{\p}{\p t} \left(\dfrac{\partial \widetilde{u}_{R}^\es(t,X)}{\partial t}-2c\dfrac{\partial \widetilde{u}_{R}^\es(t,X)}{\partial{X_2}}\right) \widetilde{\phi}_{R}(t,X)\di{X\di{t}}\\
       +\int_0^T\int_{\oll}\widetilde{D}_L\nabla_X \widetilde{u}_L^\es\cdot\nabla_X  \widetilde{\phi}_L \di{X}\di{t}+\dfrac{1}{\es}\int_0^T\int_{\omm}\left(\widetilde{D}_M^\es\nabla_X \widetilde{u}_M^\es\right)\cdot\nabla_X  \widetilde{\phi}_M\di{X}\di{t}\\
        +\int_0^T\int_{\orr} \widetilde{D}_R\nabla_X \widetilde{u}_R^\es\cdot\nabla_X  \widetilde{\phi}_R \di{X}\di{t}=\int_0^T\int_{\oll}\widetilde{f}_L \widetilde{\phi}_L\di{X}\di{t}.
\end{multline*}

To determine the energy estimates on the moving coordinates, we choose $\widetilde{\phi}=\mathds{1}_{(0,t)}\dfrac{\p \widetilde{u}}{\p t}$.

Observe that, making use of integration by parts and the periodicity in the vertical direction, we have 
\begin{align*}
    \int_0^t\int_{\oll}\left(\dfrac{\p^2\widetilde{u}_{L}^\es}{\p X_2\p t} \right) \dfrac{\p \widetilde{u}_L}{\p t} \di{X}\di{t}&=\dfrac{1}{2}\int_0^t\int_{\oll} \dfrac{\p}{\p X_2} \left(\dfrac{\p \widetilde{u}_L^\es}{\p t}\right)^2\di{X}\di{t}\\
    &=0.
\end{align*}
Similarly,
\begin{align*}
    \int_0^t \int_{\orr}\left(\dfrac{\p^2\widetilde{u}_{R}^\es}{\p X_2\p t} \right) \dfrac{\p \widetilde{u}_R}{\p t} \di{X}\di{t}&=0,
\end{align*}
and we have
\begin{align*}
   \int_0^t\int_{\omm}\rho_M^\es \dfrac{\p^2\widetilde{u}_{M}^\es}{\p X_2\p t}\dfrac{\p \widetilde{u}_{M}^\es}{\p t} \di{X}\di{t}+ \int_0^t\int_{\omm}\dfrac{\p}{\p X_2}\left(\rho_M^\es\dfrac{\p \widetilde{u}_M^\es}{\p t}\right)\dfrac{\p \widetilde{u}_{M}^\es}{\p t} \di{X}\di{t}=0.
\end{align*}

Therefore, $\widetilde{\phi}=\mathds{1}_{(0,t)}\dfrac{\p \widetilde{u}}{\p t}$ yields
\begin{multline*}
	    {\rho_L}\int_0^t\int_{\oll}\dfrac{\p^2 \widetilde{u}_{L}^\es}{\p t^2}\dfrac{\p \widetilde{u}_{L}^\es}{\p t} \di{X}+\dfrac{1}{\es}\int_0^t\int_{\omm}\rho_M^\es  \dfrac{\p^2 \widetilde{u}_{M}^\es}{\p t^2} \dfrac{\p \widetilde{u}_{M}^\es}{\p t} \di{X}
        +\rho_R\int_0^t\int_{\orr} \dfrac{\p^2 \widetilde{u}_{R}^\es}{\p t^2}\dfrac{\p \widetilde{u}_{R}^\es}{\p t}\di{X}\\
        +\int_0^t\int_{\oll}(\widetilde{D}_L \nabla_X \widetilde{u}_L^\es)\cdot\nabla_X \left(\dfrac{\p \widetilde{u}_{L}^\es}{\p t}\right) \di{X}+\dfrac{1}{\es}\int_0^t\int_{\omm}\left(\widetilde{D}_M^\es \nabla_X \widetilde{u}_M^\es\right)\cdot\nabla_X \left(\dfrac{\p \widetilde{u}_{M}^\es}{\p t}\right)\di{X}\\
        +\int_0^t\int_{\orr}(\widetilde{D}_R \nabla_X \widetilde{u}_R^\es)\cdot\nabla_X \left(\dfrac{\p \widetilde{u}_{R}^\es}{\p t}\right) \di{X}=\int_0^t\int_{\oll}f_L\dfrac{\p \widetilde{u}_{L}^\es}{\p t}\di{X}
	\end{multline*}
Since we assumed $D_{L}-c^2\rho_{L}>\lambda_L$, $D_{R}-c^2\rho_{R}>\lambda_R$ and $D_M^\es-\rho_M^\es c^2>\lambda_M$, we get  $\widetilde{D}_L, \widetilde{D}_M^\es$ and $\widetilde{D}_R$ as uniformly positive definite matrices. Now, following a similar strategy as in the proof of \Cref{lemma:energy_estimate_for_case_1}, we obtain the energy estimates 
\begin{subequations}
\begin{align}
     \left\|\ \widetilde{u}_L^\es\right\|_{L^{\infty}(0,T;L^2(\oll))}+  \left\|\ \widetilde{u}_R^\es\right\|_{L^{\infty}(0,T;L^2(\orr))}+  \dfrac{1}{\sqrt{\es}}\left\|\ \widetilde{u}_M^\es\right\|_{L^{\infty}(0,T;L^2(\omm))}&\leq C,\label{eq:energy_estimate_1_in_moving_coordinate}\\
	    \left\|\ \dfrac{\partial \widetilde{u}_L^\es}{\partial t}\right\|_{L^{\infty}(0,T;L^2(\oll))}+  \left\|\ \dfrac{\partial \widetilde{u}_R^\es}{\partial t}\right\|_{L^{\infty}(0,T;L^2(\orr))}+  \dfrac{1}{\sqrt{\es}}\left\|\ \dfrac{\partial \widetilde{u}_M^\es}{\partial t}\right\|_{L^{\infty}(0,T;L^2(\omm))}&\leq C,\\
	    \|\nabla_X \widetilde{u}_L^\es\|_{L^{\infty}(0,T;L^2(\oll))}+\|\nabla_X\widetilde{u}_R^\es\|_{L^{\infty}(0,T;L^2(\orr))}+\dfrac{1}{\sqrt{\es}}\|\nabla_X\widetilde{u}_M^\es\|_{L^{\infty}(0,T;L^2(\omm))}&\leq C,\label{eq:energy_estimate_3_in_moving_coordinate}
	\end{align}
    \end{subequations}
    where $C>0$ is an $\es$-independent constant. From the definition of $\widetilde{u}_{L},\widetilde{u}_{M}$  and $\widetilde{u}_{R}$, for $i\in\{L,M,R\}$,  we have the following norm estimates
    \begin{subequations}

    \begin{align}
          \left\| \widetilde{u}_i^\es\right\|_{L^2(\Omega_i^\es)}&=  \left\| {u}_i^\es\right\|_{L^2(\Omega_i^\es)}\label{eq:norm_estimate_1}\\
          \left\| \nabla_X \widetilde{u}_i^\es\right\|_{L^2(\Omega_i^\es)}&=  \left\|\nabla {u}_i^\es\right\|_{L^2(\Omega_i^\es)}
    \end{align}
   likewise, using the chain rule, we get
    \begin{align}
        \left\|\frac{\p {u}_i^\es}{\p t}\right\|_{L^2(\Omega_i^\es)}&\le \left\|\frac{\p \widetilde{u}_i^\es}{\p t}\right\|_{L^2(\Omega_i^\es)} + |c|\left\| \nabla_X \widetilde{u}_i^\es\right\|_{L^2(\Omega_i^\es)}\label{eq:norm_estimate_2}.
    \end{align}
    \end{subequations}
    Finally, using the norm estimates \eqref{eq:norm_estimate_1}--\eqref{eq:norm_estimate_2} and the energy estimates \eqref{eq:energy_estimate_1_in_moving_coordinate}--\eqref{eq:energy_estimate_3_in_moving_coordinate}, we establish the energy estimates \eqref{eq:energy_estimate_1_case_1}--\eqref{eq:energy_estimate_3_case_1}.
	\end{proof}

\subsubsection{Discussion on nonstability}
   Let $p^\es(t):=\left(1-\frac{2}{5}\sin{\frac{2t}{\es}}-\frac{1}{100}(1-\cos{\frac{2t}{\es}})^2\right)$.  It is well known that for a wave equation of type 
        \begin{align*}
            \frac{\p^2 v}{\p t^2}-p^\es(t)\Delta v=0 \qquad\mbox{in}\qquad (0,T)\times\Omega
        \end{align*}
       with periodic boundary condition and the initial data $v^\es(0)=0$ and $\frac{d v^\es}{dt}(0)=\sin{\frac{x_1}{\es}}$, the energy blows up as $\es\rightarrow 0$ (see \cite{colombini1984hyperbolic} for details). Similarly, one may expect that by choosing $\rho_M^\es\equiv 1$ and $D_M^\es=p^\es$ we can construct an initial condition such that the uniform energy estimate \eqref{eq:energy_estimate_1_case_1}--\eqref{eq:energy_estimate_3_case_1} fails to hold. Proving such a result requires an  additional detailed technical analysis of how the initial conditions on the bulk domain affect the oscillation and non-stability in the layer through the transmission condition. In this paper, our main aim is to derive the effective model, rather than establishing the sharp conditions for the uniform energy estimates;  such analysis falls beyond the scope of the present work.
\subsection{Multi-scale convergence for thin layer}
	The subdomain $\omm$ has a heterogeneous structure in multiple scales and has a thickness of  $\mathcal{O}(\es)$. In order to derive the homogenization limit, we require a homogenization framework that accounts for multiple spatio-temporal scale homogenization together with dimension reduction. In this section, inspired by the classical multiscale convergence \cite{allaire1996multiscale} and the two-scale convergence for thin layers \cite{maruvsic2000two,neuss2007effective},  we introduce the concept of the weak multi-scale convergence for thin layers.
	\begin{definition}\label{def:multiscale_convergence}
A sequence $v^\es\in L^2(0,T;L^2(\omm))$ is said to be multi-scale convergent (with three spatial and two temporal scales) to $v^0(t,\overline{x},\tau,y,z)\in L^2((0,T)\times \Sigma\times  \T \times Y\times Z)$, if
\begin{multline*}
    \lim_{\es\rightarrow 0}\dfrac{1}{\es}\int_0^T\int_{\omm}v^\es(t,x)\phi\left(t,\overline{x},\dfrac{t}{\es},\dfrac{x}{\es},\dfrac{x_2}{\es^2}\right)\di{t}\di{x}\\
=\mint v^0(t,\overline{x},\tau,y,z)\phi(t,\overline{x},\tau,y,z)\dzdy
\end{multline*}
for all $\phi \in L^2((0,T)\times \Sigma;\mathcal{C}_{\sharp}(\overline{\T\times Y\times Z}))$.
\newline

\begin{remark}
    The notion of multiscale convergence for thin layers can be extended naturally to an arbitrary number of well-separated spatial and temporal scales.
\end{remark}
	\end{definition}

Similar to the concept of classical multiscale convergence, we have an associated compactness result related to the multiscale convergence for thin layers.

	\begin{theorem}\label{theorem:compactness_result}
	Let $v^\es\in L^2(0,T;L^2(\omm))$ be a sequence of functions satisfying
	\begin{align}
	    \dfrac{1}{\sqrt{\es}}\|v^\es\|_{L^2(0,T;L^2(\omm))}\leq C\label{eq:l2bound_on_v_epsilon}
	\end{align}
	for some $\es$-independent constant $C>0$, then up to a subsequence (still denoted by $v^\es$), $v^\es$ converges to some $v^0\in L^2((0,T)\times \Sigma\times \T\times Y\times Z)$ in the sense of Definition \ref{def:multiscale_convergence}.
	\end{theorem}
	The proof of the aforementioned Theorem follows an argument similar to \cite[Proposition 4.2]{neuss2007effective} but extends to multiple scales. For completeness, we give the proof in the Appendix.

    \subsection{Convergence in the bulk regions and on the interfaces}
 To find the effective model, we require compactness results of the microscopic solution in the bulk region, in the layer as well as on the interfaces.   This section focuses on finding the limit of the solutions to the microscopic problem in the bulk regions $\oll$ and $\orr$, and on the interfaces $L^\es$ and $R^\es$. In the bulk region, we have the following compactness result.
      \begin{theorem}\label{theorem:convergence in the bulk}
    Assume $(A1)-(A2)$ and \eqref{eq:energy_estimate_1_case_1}--\eqref{eq:energy_estimate_3_case_1} hold. Then for $i\in\{L,R\},$ there exists $u_i^0\in L^2(0,T;\mathcal{H}_{\sharp}(\Omega_i))\cap H^1(0,T;L^2(\Omega_i))$  such that up to a subsequence
    \begin{subequations}  
    \begin{align*}
        \mathds{1}_{\oii}u_i^\es \longrightarrow u_i^0 \qquad\mbox{strongly in}\qquad L^2((0,T)\times\Omega_i),\\
        \mathds{1}_{\oii}\nabla u_i^\es \longrightarrow \nabla u_i^0 \qquad\mbox{weakly in}\qquad L^2((0,T)\times\Omega_i),\\
        \mathds{1}_{\oii}\dfrac{\partial u_i^\es}{\partial t}  \longrightarrow \dfrac{\partial u_i^0}{\partial t} \qquad\mbox{weakly in}\qquad L^2((0,T)\times\Omega_i).
    \end{align*}
    Moreover we have $u_L|_{\Gamma_L}\equiv u_R|_{\Gamma_R}.$
     \end{subequations}
   
    \end{theorem}
    
\begin{proof}
    We refer to \cite[Proposition 2.1]{neuss2007effective} for the detailed proof.
\end{proof}

Now, by the following theorem, we establish a compactness result on the interface $L^\es$ and $R^\es$.
\begin{theorem}\label{theorem: convergence in the layer left}
     Assume $(A1)-(A2)$ and \eqref{eq:energy_estimate_1_case_1}--\eqref{eq:energy_estimate_3_case_1} hold. Then, on the interfaces connecting the bulk domain and the layer, up to a subsequence, we have the following convergence results
    \begin{multline*}
       \lim_{\es \rightarrow 0} \int_0^T\int_{L^\es} u_L^\es(t,x)\phi\left(t,\overline{x},\dfrac{t}{\es}, \dfrac{x}{\es},\dfrac{x_2}{\es^2}\right)\di{x}\di{t}\\
       =\mint u_L^0(t,\overline{x})\phi(t,\bar x,\tau,\tfrac {-1}2,y_2,z)\dzdy,
    \end{multline*}
    and 
    \begin{multline*}
       \lim_{\es \rightarrow 0} \int_0^T\int_{R^\es} u_R^\es(t,x)\phi\left(t,\overline{x},\dfrac{t}{\es}, \dfrac{x}{\es},\dfrac{x_2}{\es^2}\right)\di{x}\di{t},\\
       =\mint u_R^0(t,\overline{x})\phi(t,\bar x, \tau,\tfrac 12,y_2,z)\dzdy,
    \end{multline*}
    for any $\phi \in L^2((0,T)\times \Sigma;\mathcal{C}_{\sharp}(\overline{\T\times Y\times Z}))$.
\end{theorem}
\begin{proof}
 The proof is analogous to the proof of \cite[Proposition 2.1]{neuss2007effective}. 
We define the map $w:\Omega_L\rightarrow\oll$ such that $w(x):=\left((1-\frac{\es}{2\ell})x_1-\frac{\es}{2},x_2\right)$, it follows that $w(-\ell,x_2)=(-\ell,x_2)$ and $w(0,x_2)=(-\frac{\es}{2},x_2).$  Define $\widetilde{u}_L^\es(t,x):={u}_L^\es(t,w(x))$. From the energy estimate and the \Cref{theorem:convergence in the bulk}, we have
\begin{align}\label{eq:strong_convergence_of_ultilde}
    \widetilde{u}_L^\es\rightarrow u_L^0 \qquad\mbox{strongly in } \quad L^2((0,T);L^2(\Omega_L)),
\end{align}
and  $\nabla \widetilde{u}_L^\es$ uniformly bounded in $L^2((0,T);L^2(\Omega_L)).$
Using the trace interpolation inequality, we have 
\begin{align*}
    \|\widetilde{u}_L^\es-u_L^0\|_{L^2((0,T);L^2(\Sigma))}^2\leq C\|\widetilde{u}_L^\es-u_L^0\|_{L^2((0,T);L^2(\Omega_L))}\|\widetilde{u}_L^\es-u_L^0\|_{L^2((0,T);H^1(\mol))}.
\end{align*}
Applying \eqref{eq:strong_convergence_of_ultilde} on the above inequality we arrive at
\begin{align*}
    \widetilde{u}_L^\es\rightarrow u_L^0 \qquad\mbox{strongly in } \quad L^2((0,T);L^2(\Sigma)).
\end{align*}
Furthermore, using the definition of $\widetilde{u}_L^\es$, we get
\begin{align*}
    \widetilde{u}_L^\es\left(t,\frac{-\es}{2},x_2\right)\rightarrow u_L^0 (t,0,x_2)\qquad\mbox{strongly in } \quad L^2((0,T);L^2((0,h))).
\end{align*}
That leads to
\begin{align*}
     &\!\!\!\!\lim_{\es \rightarrow 0} \int_0^T\int_{L^\es} u_L^\es(t,x)\phi\left(t,\overline{x},\dfrac{t}{\es}, \dfrac{x}{\es},\dfrac{x_2}{\es^2}\right)\di{x}\di{t}\\
     &= \lim_{\es \rightarrow 0} \int_0^T\int_{0}^h u_L^\es(t,-\frac{\es}{2},x_2)\phi\left(t,\bar x,\dfrac{t}{\es},-\frac{1}{2}, \dfrac{x_2}{\es},\dfrac{x_2}{\es^2}\right)\di{x}\di{t}\\
      &=\lim_{\es \rightarrow 0} \int_0^T\int_{0}^h \Big[u_L^\es\Big(t,-\frac{\es}{2},x_2\Big)-u_L^0(t,0,x_2)\Big]\phi\left(t,\bar x,\dfrac{t}{\es},-\frac{1}{2}, \dfrac{x_2}{\es},\dfrac{x_2}{\es^2}\right)\di{x}\di{t}\\
      &\qquad+\lim_{\es \rightarrow 0} \int_0^T\int_{0}^h u_L^0(t,0,x_2)\phi\left(t,\bar x,\dfrac{t}{\es},-\frac{1}{2}, \dfrac{x_2}{\es},\dfrac{x_2}{\es^2}\right)\di{x}\di{t}\\
      &=\mint u_L^0(t,\overline{x})\phi(t,\bar x,\tau,\tfrac {-1}2,y_2,z)\dzdy.
\end{align*}

\end{proof}
\subsection{Convergence in the layer}

We now turn to establish the compactness result in the layer. By using the energy estimates and the concept of multiscale convergence for thin layer, we show that the function $u_M^\es$ and its spatial and temporal derivatives converge to  the corresponding limit functions that capture the macroscopic behavior as well as the oscillations associated with the different microscopic scales.
   \begin{theorem}\label{theorem:convergence_in_the_layer}
    Assume $(A1)-(A2)$ and \eqref{eq:energy_estimate_1_case_1}--\eqref{eq:energy_estimate_3_case_1} hold. Then there exist \\ $u_M^0\in L^2((0,T);H^1_{\sharp}(\Sigma))\cap H^1((0,T);L^2(\Sigma))$, $u_M^1\in L^2((0,T)\times \Sigma;\mathcal{H}_{\sharp,0}(\T\times Y))$ and $u_M^2\in L^2((0,T)\times \Sigma \times \T\times Y;H^1_{\sharp,0}(Z) )$ such that up to a subsequence
    \begin{subequations}
    \begin{align}
        u_M^\es(t,x)&\longrightarrow u_M^0(t,\overline{x})\label{eq:multiscale_convergence_for_u}\\
        \nabla u_M^\es(t,x)&\longrightarrow \nabla_{\overline{x}}u_M^0(t,\overline{x})+\nabla_y u_M^1(t,\overline{x},\tau,y)+\frac{d}{dz}u_M^2(t,\overline{x},\tau,y,z)e_2\label{eq:multiscale_convergence_for_gradient_term} \\
        \dfrac{\partial u_M^\es(t,x)}{\partial t}&\longrightarrow  \dfrac{\partial u_M^0(t,x)}{\partial t}+ \dfrac{\partial u_M^1(t,\overline{x},\tau,y)}{\partial \tau}\label{eq:multiscale_convergence_for_time_derivative_term}
    \end{align}
    \end{subequations}
    where the convergences are in the sense of multi-scale convergence for the thin layer.
    \end{theorem}
    \begin{proof}
           From the uniform energy estimates and the Theorem \ref{theorem:compactness_result}, there exist $u_M^0,\xi_0,\xi_1\in L^2((0,T)\times \Sigma\times \T\times Y\times Z)$ such that
        \begin{align*}
            u_M^\es(t,x)&\longrightarrow u_M^0(t,\overline{x},\tau,y,z),\\
        \nabla u_M^\es(t,x)&\longrightarrow \xi_0(t,\overline{x},\tau,y,z),
        \end{align*}
        and
        \begin{align*}
        \dfrac{\partial u_M^\es}{\partial t}&\longrightarrow  \xi_1(t,\overline{x},\tau,y,z),
        \end{align*}
        in the sense of multi-scale convergence for the thin layer. We first show that the function $u_M^0$ is independent of the variables $\tau,y$ and $z$. 
        
        Let $\Phi_1\in C_c^\infty((0,T)\times\Sigma;C_{\sharp}^\infty(Z))^2$, then by using the integration by parts with respect to the space variable, we get
        \begin{align*}
            0&=\lim_{\es\rightarrow 0}\es\int_0^T\int_{\omm}\nabla u_M^\es(t,x)\cdot\Phi_1(t,\overline{x},\dfrac{x_2}{\es^2})\di{x}\di{t}\\
            &=-\lim_{\es\rightarrow 0}\es\int_0^T\int_{\omm} u_m^\es(t,x)\nabla_x\cdot \Phi_1(t,\overline{x},\dfrac{x_2}{\es^2})\di{x}\di{t}-\lim_{\es\rightarrow 0}\frac{1}{\es}\int_0^T\int_{\omm} u_m^\es(t,x)\nabla_z\cdot \Phi_1(t,\overline{x},\dfrac{x_2}{\es^2})\di{x}\di{t}\\
            &=-\mint u_M^0(t,\tau,x,y,z)\nabla_z\cdot \Phi_1(t,\overline{x},z)\di{z}\di{y}\di{\tau}\di{x}\di{t},
        \end{align*}
        which shows that $u_M^0$ is independent of $z$.  Similarly, let $\Phi_2\in  C_c^\infty((0,T)\times\Sigma;C_{\sharp}^\infty(Y))^2$ and $\phi_3\in  C_c^\infty((0,T)\times\Sigma;C_{\sharp}^\infty(\T))$, we get
        \begin{align*}
            0&=\lim_{\es\rightarrow 0}\int_0^T\int_{\omm}\nabla u_m^\es(t,x)\cdot \Phi_2(t,\overline{x},\dfrac{x}{\es})\di{x}\di{t}\\
            &=-\lim_{\es\rightarrow 0}\es\frac{1}{\es}\int_0^T\int_{\omm} u_m^\es(t,x)\nabla_x\cdot \Phi_2(t,\overline{x},\dfrac{x}{\es})\di{x}\di{t}-\lim_{\es\rightarrow 0}\frac{1}{\es}\int_0^T\int_{\omm} u_m^\es(t,x)\nabla_y\cdot\Phi_2(t,\overline{x},\dfrac{x}{\es})\di{x}\di{t}\\
            &=-\int_0^T\int_{\Sigma}\int_\T\int_{Y}u_M^0(t,\tau,x,y)\nabla_y\cdot\Phi_2(t,\overline{x},y)\di{y}\di{\tau}\di{x}\di{t},
        \end{align*}
        and 
        \begin{align*}
            0&=\lim_{\es\rightarrow 0}\int_0^T\int_{\omm}\dfrac{\partial u_M^\es}{\partial t}(t,x)\phi_3(t,\overline{x},\dfrac{t}{\es})\di{x}\di{t}\\
            &=-\lim_{\es\rightarrow 0}\es\int_0^T\int_{\omm} u_m^\es(t,x)\partial_t\phi_3(t,\overline{x},\dfrac{t}{\es})\di{x}\di{t}-\lim_{\es\rightarrow 0}\frac{1}{\es}\int_0^T\int_{\omm} u_m^\es(t,x)\partial_\tau \phi_3(t,\overline{x},\dfrac{t}{\es})\di{x}\di{t}\\
            &=-\int_0^T\int_{\Sigma}\int_\T u_M^0(t,\tau,x)\partial_\tau \phi_3(t,\overline{x},\tau)\di{\tau}\di{x}\di{t},
        \end{align*}
        that shows $u_M^0$ is independent of $y$ and $\tau$.

        Now, we derive the multiscale convergence limit of $\frac{\p u_M^\es}{\p t}$ and $\nabla u_M^\es$ as $\es\rightarrow 0$.
        Let $\Phi:=(\phi_1,\phi_2,\phi_3)\in C_c^\infty ((0,T)\times \Sigma;C_{\sharp}(\T\times Y\times Z))^3$ and  $\Phi_1:=(\phi_2,\phi_3)$, we define 
        \begin{align*}
            \nabla_{t,\bar x}\cdot\Phi&:=\frac{\p}{\p t}\phi_1+\nabla_{\bar x}\cdot \Phi_1,\\
            \nabla_{\tau,y}\cdot\Phi&:=\frac{\p}{\p \tau}\phi_1+\nabla_{y}\cdot \Phi_1.
        \end{align*}
        Assume that the function $\Phi$ satisfies 
        \begin{align}\label{eq:derivative_of_phi3}
            \frac{d}{dz}\phi_3=0 \quad\mbox{and }\quad \nabla_{\tau,y}\cdot\int_Z\Phi\di{z}=0.
        \end{align} 
        Now, using integration by parts, we have
        \begin{align}\label{eq:nablatx_id}
            \frac{1}{\es}\int_0^T\int_{\omm}\nabla_{t,x} u_M^\es\cdot \Phi\di{x}\di{t}=-\frac{1}{\es}\Mint u_M^\es\Big[\nabla_{t,\bar x}\cdot\Phi+\frac{1}{\es}\nabla_{\tau,y}\cdot\Phi+\frac{1}{\es^2}\nabla_z\cdot \Phi_1\Big]\di{x}\di{t}.        \end{align}
            From \eqref{eq:derivative_of_phi3}, we have $\nabla_z\cdot \Phi_1=0$, using Lemma \ref{lemma:h-1bound} together with \eqref{eq:derivative_of_phi3}, we obtain
            \begin{align*}
                \Bigl\|\frac{1}{\es^2}\nabla_{\tau,y}\cdot\Phi\Bigr\|_{L^2(0,T;H^{1}(\omm)^*)}\leq C.
            \end{align*}
           Since we have the uniform energy  bound
           \begin{align*}
               \dfrac{1}{\es}\| u_M^\es\|_{L^2(0,T;H^1(\omm))}^2\leq C,
           \end{align*}
           we obtain
            \begin{align*}
                \lim_{\es\rightarrow 0}\es\frac{1}{\es}\Mint u_M^\es \frac{1}{\es^2}\nabla_{\tau,y}\cdot \Phi \di{x}\di{y}=0.
            \end{align*}

        Using the preceding identity and  taking the limit $\es\rightarrow 0$ in \eqref{eq:nablatx_id} yields
            \begin{align*}
                \mint(\xi_0,\xi_1)\cdot\Phi\dzdy=\mint u_M^0 \nabla_{t,\bar x}\cdot \Phi\dzdy.
            \end{align*}
           Performing integration by parts with respect to the space and time variables, we obtain

        \begin{align*}
            \mint[(\xi_0,\xi_1)-\nabla_{t,\bar x}u_M^0]\cdot\Phi\dzdy=0.
        \end{align*}
        Since the function $\Phi$ satisfies the conditions \eqref{eq:derivative_of_phi3}, using the above identity together with \cite[Lemma 3.7]{allaire1996multiscale}, we get that there exists $u_M^1\in L^2((0,T)\times \Sigma;H^1_{\sharp}(\T\times Y))$ and $u_M^2\in L^2((0,T)\times \Sigma \times \T\times Y;H^1_{\sharp}(Z))$ such that
        \begin{align*}
            (\xi_0,\xi_1)-\nabla_{t,\bar x}u_M^0=\nabla_{\tau,y}u_M^1+[0,\nabla_z]^Tu_M^2.
        \end{align*}
        By rearranging, we get
        \begin{align*}
            \xi_0(t,\bar x,\tau,y,z)&=\nabla_{\bar x}u_M^0(t,\bar x)+\nabla_{y}u_M^1(t,\bar x,\tau, y)+\frac{d}{dz}u_M^2(t,\bar x,\tau, y,z)e_2\\
            \xi_1(t,\bar x,\tau,y,z)&=\frac{\p}{\p t}u_M^0(t,\bar x)+\frac{\p}{\p \tau}u_M^1(t,\bar x,\tau, y).
        \end{align*}
    \end{proof}
    \section{Upscaling of the microscopic model}\label{section:Upscaling of the microscopic model}
    This section is devoted to deriving the effective model and the interface conditions. Relying on the compactness results presented in the \Cref{section:Uniform energy estimates and compactness results}, and using  well-constructed test functions in the microscopic weak formulation, we pass to the limit $\es\rightarrow0$ and identify the homogenized limit equation. 
   Furthermore, under certain additional assumptions on the effective coefficients, such as positivity and boundedness, we obtain the uniqueness result for the effective model.
    
    \subsection{Derivation of the effective model}
    In the following lemma, we  derive the first effective interface condition. We show that the limits $u_L^0$ and $u_R^0$ that are defined in the bulk domains, and the limit $u_M^0$ on the interface $\Sigma$ coincide on $\Sigma$. In other words, we get the continuity of the solution along the interface $\Sigma$.

    \begin{lemma}\label{lemma:first_transmission_condition}
        Let $u_L^0,u_R^0$ and $u_M^0$ be the limit functions from  \Cref{theorem:convergence in the bulk} and \Cref{theorem:convergence_in_the_layer}. Then, it holds that
        \begin{align*}
            {u_R^0=u_L^0}&{=u_M^0}\qquad &\mbox{for a.e. }(t,\bar x)\in (0,T)\times\Sigma.
        \end{align*}
    \end{lemma}

     The proof follows an argument similar to that in \cite[Theorem 5.3]{effectiveapratim} with minor modifications. For the reader’s convenience, we give the details in the Appendix.

   In the next lemma, using the weak convergence results in the bulk domain and the multiscale convergence results in the layer, we determine the homogenization limit and the second effective transmission condition.
\begin{lemma}\label{lemma:derivation_of_uL_n_uR}
     Assume $(A1)-(A2)$ and \eqref{eq:energy_estimate_1_case_1}--\eqref{eq:energy_estimate_3_case_1} hold. Then, in the distributional sense, the functions $u_L^0,u_R^0,u_M^0,u_M^1$ and $u_M^2$ from \Cref{theorem:convergence in the bulk} and \Cref{theorem:convergence_in_the_layer} satisfy
 \begin{subequations}  
\begin{align}
 \dfrac{\partial}{\partial t}\left( {\rho_L}\dfrac{\partial u_L^0}{\partial t}\right)-\mathrm{div}(D_L \nabla u_L^0)&=f_L\qquad &\mbox{in}\qquad (0,T)\times \Omega_L\label{eq:macro_ul}\\
  \dfrac{\partial}{\partial t}\left({\rho_R} \dfrac{\partial u_R^0}{\partial t}\right)-\mathrm{div}(D_R \nabla u_R^0)&=0\qquad &\mbox{in}\qquad (0,T)\times\Omega_R
\end{align}
\begin{multline}\label{eq:transmition_condition_2}
     {(D_R \nabla u_R^0)\cdot n-(D_L \nabla u_L^0)\cdot n}\\
     =\dfrac{\p}{\p t}\int_\T\int_{Y}\int_Z\left(\rho_M(t,\overline{x},\tau,y,z)\left(\dfrac{\partial u_M^0(t,x)}{\partial t}+ \dfrac{\partial u_M^1(t,\overline{x},\tau,y)}{\partial \tau}\right)\right)\di{z}\di{y}\di{\tau}\\
     -\mathrm{div}_{\bar x}\int_\T\int_{Y}\int_ZD_M(t,\overline{x},\tau,y,z) \left( \nabla_{\overline{x}}u_M^0(t,\overline{x})+\nabla_y u_M^1(t,\overline{x},\tau,y)+\nabla_zu_M^2(t,\overline{x},\tau,y,z)\right)\di{z}\di{y}\di{\tau}\\
      \qquad \mbox{on}\qquad(0,T)\times\Sigma.
\end{multline}

\end{subequations}
\end{lemma}
    \begin{proof}
         In the weak formulation of the microscopic problem, choose the test function $\phi=\phi_0 \in C_c^\infty ((0,T)\times \Omega)$. We get 
\begin{multline*}
    -\rho_L\int_0^T\int_{\oll}  \dfrac{\partial u_L^\es}{\partial t} \dfrac{\partial \phi_0}{\partial t} \di{x}-  \dfrac{1}{\es} \int_0^T\int_{\omm}  \left({\rho_M^\es} \dfrac{\partial u_M^\es}{\partial t}\right) \dfrac{\partial \phi_0}{\partial t}\di{x} -    \rho_R\int_0^T\int_{\orr} \dfrac{\partial u_R^\es}{\partial t} \dfrac{\partial \phi_0}{\partial t}\di{x}\\
    +D_L\int_0^T\int_{\oll}\nabla u_L^\es \cdot \nabla \phi_0   \di{x} + D_R\int_0^T\int_{\orr} \nabla u_R^\es\cdot\nabla \phi_0 \di{x} 
    + \dfrac{1}{\es}\int_0^T\int_{\omm}  D_M^\es  \nabla u_M^\es\cdot\nabla \phi_0 \di{x}\\
    = \int_0^T\int_{\oll} f_L\phi_0 \di{x}
\end{multline*}
Now, using Theorem \ref{theorem:convergence in the bulk} and Theorem \ref{theorem:convergence_in_the_layer} and taking the limit as $\es\rightarrow 0$, we obtain 
\begin{multline*}
    -\rho_L\int_0^T\int_{\Omega_L}  \dfrac{\partial u_L^0(t,x)}{\partial t} \dfrac{\partial \phi_0(t,x)}{\partial t} \di{x}\di{t}-    \rho_R\int_0^T\int_{\Omega_R} \dfrac{\partial u_R^0(t,x)}{\partial t} \dfrac{\partial \phi_0(t,x)}{\partial t}\di{x}\di{t}\\
    -  \int_0^T\int_{\Sigma}\int_\T\int_{Y}\int_Z  {\rho_M(t,\overline{x},\tau,y,z)} \left(\dfrac{\partial u_M^0(t,x)}{\partial t}+ \dfrac{\partial u_M^1(t,\overline{x},\tau,y)}{\partial \tau}\right) \dfrac{\partial \phi_0(t,x)}{\partial t}\di{z}\di{y}\di{\tau}\di{x}\di{t}\\
    +D_L\int_0^T\int_{\Omega_L}\nabla u_L^0(t,x) \cdot \nabla \phi_0(t,x)   \di{x} + D_R\int_0^T\int_{\Omega_R} \nabla u_R^0(t,x)\cdot\nabla \phi_0(t,x) \di{x}\di{t}\\ 
    +  \int_0^T\int_{\Sigma}\int_\T\int_{Y}\int_Z   D_M(t,\overline{x},\tau,y,z) \left( \nabla_{\overline{x}}u_M^0(t,\overline{x})+\nabla_y u_M^1(t,\overline{x},\tau,y)\right.\\
\left.+\nabla_zu_M^2(t,\overline{x},\tau,y,z)\right)\cdot \nabla \phi_0(t,x) \di{z}\di{y}\di{\tau}\di{x}\di{t}\\
    = \int_0^T\int_{\Omega_L} f_L(t,x)\phi_0(t,x) \di{x}\di{t}.
\end{multline*}

which is the weak form of the system \eqref{eq:macro_ul}--\eqref{eq:transmition_condition_2}.

    \end{proof}

    The preceding lemma derived the transmission condition involving the functions $u_M^0$, $u_M^1$, and $u_M^2$. However, this condition cannot be evaluated explicitly without further characterization of these functions. Therefore, in the following lemma, using well-constructed test functions, we derive a system of partial differential equations that will later allow us to represent $u_M^1$ and $u_M^2$ in terms of $u_M^0$ and suitable cell problems of elliptic and hyperbolic type.
    \begin{lemma}
   Assume $(A1)-(A2)$ and \eqref{eq:energy_estimate_1_case_1}--\eqref{eq:energy_estimate_3_case_1} hold. Then, in the distributional sense, the functions $u_M^0,u_M^1$ and $u_M^2$ from  \Cref{theorem:convergence_in_the_layer} satisfy
 \begin{subequations}  

\begin{multline}\label{eq:strong_form_u_1}
   \dfrac{\p}{\p \tau} \int_Z \left( {\rho_M(t,\overline{x},\tau,y,z)} \left(\dfrac{\partial u_M^0(t,x)}{\partial t}+ \dfrac{\partial u_M^1(t,\overline{x},\tau,y)}{\partial \tau}\right)\right)\di{z}\\
     -\mathrm{div_y}\int_Z \bigg( D_M(t,\overline{x},\tau,y,z)
    \left( \nabla_{\overline{x}}u_M^0(t,\overline{x})+\nabla_y u_M^1(t,\overline{x},\tau,y)+\nabla_zu_M^2(t,\overline{x},\tau,y,z)\right)\bigg)\di{z}=0\\
    \mbox{for a.e.} (t,\overline{x},\tau,y)\in (0,T)\times\Sigma\times (0,1)\times Y,
\end{multline}
along with the boundary condition
\begin{align}\label{eq:strong_form_u_1_bc}
    \bigg( D_M(t,\overline{x},\tau,y,z)
    \left( \nabla_{\overline{x}}u_M^0(t,\overline{x})+\nabla_y u_M^1(t,\overline{x},\tau,y)+\nabla_zu_M^2(t,\overline{x},\tau,y,z)\right)\bigg)\cdot n=0\quad\mbox{on}\quad Y_L\cup Y_R,
\end{align}
and
\begin{multline}\label{eq:strong_form_for_u_2}
    -\mathrm{div_z}\Big(   D_M(t,\overline{x},\tau,y,z)
    \left( \nabla_{\overline{x}}u_M^0(t,\overline{x})+\nabla_y u_M^1(t,\overline{x},\tau,y)+\nabla_zu_M^2(t,\overline{x},\tau,y,z)\right)\Big)=0\\
    \mbox{for a.e. } (t,\overline{x},\tau,y,z)\in (0,T)\times\Sigma\times (0,1)\times Y\times Z.
\end{multline}
\end{subequations}
\end{lemma}
\begin{proof}
    We first construct a test function in $(0,T)\times\Omega$ with suitable regularity and support in both fast and slow variables, which is then used in the weak formulation of the microscopic problem. Then, using the uniform energy estimates and the multiscale convergence for thin layer, as $\es\rightarrow0$, we deduce \eqref{eq:strong_form_u_1} and \eqref{eq:strong_form_u_1_bc}.
Let
\begin{align*}
    \phi_1:=\begin{cases} 
      \alpha_1(t)\alpha_2(\frac{t}{\es})\beta(\bar x) \psi_L(\frac{\bar x}{\es})\rho(\frac{x_1}{\es}) &  \qquad \mbox{in}\qquad (0,T)\times \oll\\
      \alpha_1(t)\alpha_2(\frac{t}{\es})\beta(\bar x) \psi_M(\frac{x}{\es}) & \qquad \mbox{in}\qquad (0,T)\times \omm\\
      \alpha_1(t)\alpha_2(\frac{t}{\es})\beta(\bar x) \psi_R(\frac{\bar x}{\es})\rho(\frac{-x_1}{\es}) & \qquad \mbox{in}\qquad (0,T)\times \orr,
   \end{cases}
\end{align*}
where
\begin{itemize}
    \item $\alpha_1\in C_c^\infty((0,T))$ and $\beta\in C_c^\infty(\Sigma).$
    \item $ \alpha_2\in C_{c}^\infty([0,1])$ which extends periodically to whole $\mathbb{R}.$
    \item $\psi_M\in C_c^\infty(\overline Y) $ which extends periodically with respect to $y_2$ variable.
    \item For $i\in \{L,R\}$, $ \psi_i $ is defined in $(0,T)\times \overline{\oii}$ such that $\psi_i\equiv\psi_M$ on $i^\es$ and independent of $x_1$ variable.
    \item $\rho\in C^\infty([0,\infty))$ is a cut off function satisfying  $0\leq\rho \leq1$, $\rho=1$ in $[0,1]$ and $\rho=0 $ in $[2,\infty).$
\end{itemize}

In the weak formulation of the microscopic problem, we choose $\phi=\es\phi_1$ as a test and obtain the identity
 \begin{multline}\label{eq:phi1weakform}
   \underbrace{ \es \rho_L\int_0^T\int_{\oll}  \dfrac{\partial^2u_L^\es}{\partial t^2} (\alpha_1(t)\alpha_2(\frac{t}{\es})\beta(\bar x) \psi_L(\frac{\bar x}{\es})\rho(\frac{x_1}{\es}))\di{x}\di{t}}_{I_1}\\
    +  \underbrace{\int_0^T\int_{\omm}  \dfrac{\partial}{\partial t}\left({\rho_M^\es} \dfrac{\partial u_M^\es}{\partial t}\right) ( \alpha_1(t)\alpha_2(\frac{t}{\es})\beta(\bar x) \psi_M(\frac{x}{\es}))\di{x}\di{t}}_{I_2} \\
    +   \underbrace{ \es \rho_R\int_0^T\int_{\orr} \dfrac{\partial^2u_R^\es}{\partial t^2} ( \alpha_1(t)\alpha_2(\frac{t}{\es})\beta(\bar x) \psi_R(\frac{\bar x}{\es})\rho(\frac{-x_1}{\es}))\di{x}\di{t}}_{I_3}\\
    +\underbrace{\es D_L\int_0^T\int_{\oll}\nabla u_L^\es \cdot \nabla ( \alpha_1(t)\alpha_2(\frac{t}{\es})\beta(\bar x) \psi_L(\frac{\bar x}{\es})\rho(\frac{x_1}{\es}))   \di{x}\di{t}}_{I_4}\\
    + \underbrace{\es D_R\int_0^T\int_{\orr} \nabla u_R^\es\cdot\nabla ( \alpha_1(t)\alpha_2(\frac{t}{\es})\beta(\bar x) \psi_R(\frac{\bar x}{\es})\rho(\frac{-x_1}{\es})) \di{x}\di{t}}_{I_5} \\
    +\underbrace{ \int_0^T\int_{\omm}  D_M^\es  \nabla u_M^\es\cdot\nabla (\alpha_1(t)\alpha_2(\frac{t}{\es})\beta(\bar x) \psi_M(\frac{x}{\es})) \di{x}\di{t}}_{I_6}\\
    = \underbrace{\es \int_0^T\int_{\oll} f_L(\alpha_1(t)\alpha_2(\frac{t}{\es})\beta(\bar x) \psi_L(\frac{\bar x}{\es})\rho(\frac{x_1}{\es})) \di{x}\di{t}}_{I_7}.
\end{multline}
We study the limit of each term in the aforementioned identity separately.
Using integration by parts with respect to the time variable, we can write the term $ I_1$ as
\begin{multline*} 
    I_1=-\es \rho_L\int_0^T\int_{\oll}  \dfrac{\partial u_L^\es}{\partial t} (\alpha_1'(t)\alpha_2(\frac{t}{\es})\beta(\bar x) \psi_L(\frac{\bar x}{\es})\rho(\frac{x_1}{\es}))\di{x}\di{t}\\
    -\rho_L\int_0^T\int_{\oll}  \dfrac{\partial u_L^\es}{\partial t} (\alpha_1(t)\alpha_2'(\frac{t}{\es})\beta(\bar x) \psi_L(\frac{\bar x}{\es})\rho(\frac{x_1}{\es}))\di{x}\di{t}.
\end{multline*}
The terms in the first integral are uniformly bounded with respect to the $L^2$-norm; therefore, using the Cauchy-Schwarz inequality, the first term converges to $0$ as $\es\rightarrow0$. 
For the second term of $I_1$, we have 

\begin{align*}
    \rho_L\int_0^T\int_{\oll}  \dfrac{\partial u_L^\es}{\partial t} (\alpha_1(t)\alpha_2'(\frac{t}{\es})\beta(\bar x) \psi_L(\frac{\bar x}{\es})\rho(\frac{x_1}{\es}))\di{x}\di{t}&\leq C \int_0^T\int_{\oll} \left| \dfrac{\partial u_L^\es}{\partial t}\right| |\rho(\frac{x_1}{\es}))| \di{x}\di{t}\\
    &\leq C  \left\|\ \dfrac{\partial u_L^\es}{\partial t}\right\|_{L^{2}(0,T;L^2(\oll))}T\Big(\int_0^{2\es}\rho(\frac{x_1}{\es})^2\di{x_1}\Big)^{\frac{1}{2}}\\
    &\leq \sqrt{\es} C.
\end{align*}
Thus, the second term also converges to $0$ as $\es\rightarrow 0$, this leads to $I_1\rightarrow0 $ as $\es\rightarrow 0$. Similarly, $I_3\rightarrow0 $ as $\es\rightarrow 0$.

Next, we look at the term $I_4$, and observe that
\begin{multline*}
  I_4
    =\es D_L\int_0^T\int_{\oll}\nabla u_L^\es \cdot \alpha_1(t)\alpha_2(\frac{t}{\es})  \nabla_{ x} \beta(\bar x) \psi_L(\frac{\bar x}{\es})\rho(\frac{x_1}{\es})   \di{x}\di{t}\\
    +D_L\int_0^T\int_{\oll}\nabla u_L^\es \cdot \alpha_1(t)\alpha_2(\frac{t}{\es})  \beta(\bar x)  \nabla_{ x}\psi_L(\frac{\bar x}{\es})\rho(\frac{x_1}{\es})   \di{x}\di{t}\\
    +D_L\int_0^T\int_{\oll}\nabla u_L^\es \cdot \alpha_1(t)\alpha_2(\frac{t}{\es})  \beta(\bar x)  \psi_L(\frac{\bar x}{\es})\rho'(\frac{x_1}{\es})e_1   \di{x}\di{t}.
\end{multline*}
it follows that the first term converges to $0$ as $\es\rightarrow0$. For the second term, we have
\begin{align*}
    \left|D_L\int_0^T\int_{\oll}\nabla u_L^\es \cdot \alpha_1(t)\alpha_2(\frac{t}{\es})  \beta(\bar x)  \nabla_{ x}\psi_L(\frac{\bar x}{\es})\rho(\frac{x_1}{\es})   \di{x}\di{t}\right|&\leq C \|\nabla u_L^\es\|_{L^2(0,T;L^2(\oll))}T\Big(\int_0^{2\es}\rho(\frac{x_1}{\es})^2\di{x_1})\Big)^{\frac{1}{2}}\\
    &\leq \sqrt{\es} C,
\end{align*}
where $C$ is independent of $\es$, and for the third term, we have
\begin{align*}
    \left|D_L\int_0^T\int_{\oll}\nabla u_L^\es \cdot \alpha_1(t)\alpha_2(\frac{t}{\es})  \beta(\bar x)  \psi_L(\frac{\bar x}{\es})\rho'(\frac{x_1}{\es}) e_1  \di{x}\di{t}\right|&\leq C \|\nabla u_L^\es\|_{L^2(0,T;L^2(\oll))}T\Big(\int_0^{2\es}\rho'(\frac{x_1}{\es})^2\di{x_1})\Big)^{\frac{1}{2}}\\
    &\leq \sqrt{\es} C,
\end{align*}
where $C$ is independent of $\es$. Consequently,  as $\es\rightarrow 0,$  we get $I_4\rightarrow 0.$ In a similar way, we deduce $I_5\rightarrow 0$ as $\es\rightarrow 0$. For the term $I_7$, each term inside the integral is bounded; therefore, as $\es\rightarrow 0$, $I_7\rightarrow 0$.

Now, we look at the term $I_2$ and we have 
\begin{multline*}
    I_2=- \int_0^T\int_{\omm}  {\rho_M^\es} \dfrac{\partial u_M^\es}{\partial t} ( \alpha_1'(t)\alpha_2(\frac{t}{\es})\beta(\bar x) \psi_M(\frac{x}{\es}))\di{x}\di{t}\\
    -\frac{1}{\es} \int_0^T\int_{\omm} {\rho_M^\es} \dfrac{\partial u_M^\es}{\partial t} ( \alpha_1(t)\alpha_2'(\frac{t}{\es})\beta(\bar x) \psi_M(\frac{x}{\es}))\di{x}\di{t}.
\end{multline*}
Relying on  \Cref{def:multiscale_convergence} and \eqref{eq:multiscale_convergence_for_time_derivative_term}, we arrive at
\begin{align*}
    \lim_{\es\rightarrow 0}\int_0^T\int_{\omm}  {\rho_M^\es} \dfrac{\partial u_M^\es}{\partial t} ( \alpha_1'(t)&\alpha_2(\frac{t}{\es})\beta(\bar x) \psi_M(\frac{x}{\es}))\di{x}\di{t}\\
    &=\lim_{\es\rightarrow 0}\es\cdot \lim_{\es\rightarrow 0} \frac 1\es \int_0^T\int_{\omm}  {\rho_M^\es} \dfrac{\partial u_M^\es}{\partial t} ( \alpha_1'(t)\alpha_2(\frac{t}{\es})\beta(\bar x) \psi_M(\frac{x}{\es}))\di{x}\di{t}\\
    &= 0,
\end{align*}
and
\begin{multline*}
    \lim_{\es\rightarrow 0}\frac{1}{\es} \int_0^T\int_{\omm} {\rho_M^\es} \dfrac{\partial u_M^\es}{\partial t} ( \alpha_1(t)\alpha_2'(\frac{t}{\es})\beta(\bar x) \psi_M(\frac{x}{\es}))\di{x}\di{t}\\
    = \mint  {\rho_M(t,\overline{x},\tau,y,z)} \left(\dfrac{\partial u_M^0(t,x)}{\partial t}+ \dfrac{\partial u_M^1(t,\overline{x},\tau,y)}{\partial \tau}\right) \\ ( \alpha_1(t)\alpha_2'(\tau)\beta(\bar x) \psi_M(y))\dzdy.
\end{multline*}
It remains to show the limit of the term $I_6$. It holds that
\begin{multline*}
    I_6=\int_0^T\int_{\omm}  D_M^\es  \nabla u_M^\es\cdot (\alpha_1(t)\alpha_2(\frac{t}{\es})\nabla_{ x}\beta(\bar x) \psi_M(\frac{x}{\es})) \di{x}\di{t}\\
    +\frac1\es\int_0^T\int_{\omm}  D_M^\es  \nabla u_M^\es\cdot (\alpha_1(t)\alpha_2(\frac{t}{\es})\beta(\bar x) \nabla\psi_M(\frac{x}{\es})) \di{x}\di{t}.
\end{multline*}
Using \Cref{def:multiscale_convergence} and \eqref{eq:multiscale_convergence_for_gradient_term}, for the first term we have 
\begin{align*}
    \lim_{\es\rightarrow 0} \int_0^T\int_{\omm}  D_M^\es  \nabla u_M^\es\cdot (\alpha_1(t)&\alpha_2(\frac{t}{\es})\nabla\beta(\bar x)\psi_M(\frac{x}{\es})) \di{x}\di{t}\\
    &=\lim_{\es\rightarrow 0}\es\cdot \lim_{\es\rightarrow 0} \frac1\es \int_0^T\int_{\omm}  D_M^\es  \nabla u_M^\es\cdot (\alpha_1(t)\alpha_2(\frac{t}{\es})\nabla\beta(\bar x)\psi_M(\frac{x}{\es})) \di{x}\di{t}\\
    &= 0,
\end{align*}
and for the second term, we have 
\begin{multline*}
    \lim_{\es\rightarrow 0}\frac1\es\int_0^T\int_{\omm}  D_M^\es  \nabla u_M^\es\cdot (\alpha_1(t)\alpha_2(\frac{t}{\es})\beta(\bar x) \nabla\psi_M(\frac{x}{\es})) \di{x}\di{t}\\
    =\int_0^T\int_{\Sigma}\int_\T\int_{Y}\int_Z   D_M(t,\overline{x},\tau,y,z)
    \left( \nabla_{\overline{x}}u_M^0(t,\overline{x})+\nabla_y u_M^1(t,\overline{x},\tau,y)+\nabla_zu_M^2(t,\overline{x},\tau,y,z)\right)\\
    \cdot ( \alpha_1(t)\alpha_2(\tau)\beta(\bar x) \nabla_y\psi_M(y)) \di{z}\di{y}\di{\tau}\di{x}\di{t}.
\end{multline*}


Collecting the limits of $I_1-I_7$, as $\es\rightarrow 0$, from the identity \eqref{eq:phi1weakform} we arrive at
\begin{multline*}
     -\int_0^T\int_{\Sigma}\int_\T\int_{Y}\int_Z  {\rho_M(t,\overline{x},\tau,y,z)} \left(\dfrac{\partial u_M^0(t,x)}{\partial t}+ \dfrac{\partial u_M^1(t,\overline{x},\tau,y)}{\partial \tau}\right) \\ ( \alpha_1(t)\alpha_2'(\tau)\beta(\bar x) \psi_M(y))\di{z}\di{y}\di{\tau}\di{x}\di{t}\\
     +\int_0^T\int_{\Sigma}\int_\T\int_{Y}\int_Z   D_M(t,\overline{x},\tau,y,z)
    \left( \nabla_{\overline{x}}u_M^0(t,\overline{x})+\nabla_y u_M^1(t,\overline{x},\tau,y)+\nabla_zu_M^2(t,\overline{x},\tau,y,z)\right)\\
    \cdot ( \alpha_1(t)\alpha_2(\tau)\beta(\bar x) \nabla_y\psi_M(y)) \di{z}\di{y}\di{\tau}\di{x}\di{t}=0.
\end{multline*}

That is the weak form of the identity \eqref{eq:strong_form_u_1} and \eqref{eq:strong_form_u_1_bc}. 

We are left to show the identity \eqref{eq:strong_form_for_u_2}. To do so, we construct a new test function $\phi_2$ in $(0,T)\times\Omega$.
Let
\begin{align*}
    \phi_2:=\begin{cases} 
      \alpha_1(t)\alpha_2(\frac{t}{\es})\beta(\bar x) \psi_L(\frac{\bar x}{\es})\rho(\frac{x_1}{\es})\gamma(\frac{x_2}{\es^2}) &  \qquad \mbox{in}\qquad (0,T)\times \oll\\
      \alpha_1(t)\alpha_2(\frac{t}{\es})\beta(\bar x) \psi_M(\frac{x}{\es})\gamma(\frac{x_2}{\es^2}) & \qquad \mbox{in}\qquad (0,T)\times \omm\\
      \alpha_1(t)\alpha_2(\frac{t}{\es})\beta(\bar x) \psi_R(\frac{\bar x}{\es})\rho(\frac{-x_1}{\es})\gamma(\frac{x_2}{\es^2}) & \qquad \mbox{in}\qquad (0,T)\times \orr,
   \end{cases}
\end{align*}
where $\alpha_1,\alpha_2,\beta, \psi_L,\psi_M,\psi_R,\rho$ are functions defined as before and $\gamma\in C_c^{\infty}([0,1])$ which extends periodically to whole $\mathbb{R}$. Choosing this well-constructed test function $\phi=\es^2 \phi_2$ in the weak formulation \eqref{eq:weakform}, we get
 \begin{multline}\label{eq:phi2weakform}
   \underbrace{ \es^2 \rho_L\int_0^T\int_{\oll}  \dfrac{\partial^2u_L^\es}{\partial t^2} (\alpha_1(t)\alpha_2(\frac{t}{\es})\beta(\bar x) \psi_L(\frac{\bar x}{\es})\rho(\frac{x_1}{\es})\gamma(\frac{x_2}{\es^2}))\di{x}\di{t}}_{I_8}\\
    +  \underbrace{\es\int_0^T\int_{\omm}  \dfrac{\partial}{\partial t}\left({\rho_M^\es} \dfrac{\partial u_M^\es}{\partial t}\right) ( \alpha_1(t)\alpha_2(\frac{t}{\es})\beta(\bar x) \psi_M(\frac{x}{\es})\gamma(\frac{x_2}{\es^2}))\di{x}\di{t}}_{I_9} \\
    +   \underbrace{ \es^2 \rho_R\int_0^T\int_{\orr} \dfrac{\partial^2u_R^\es}{\partial t^2} ( \alpha_1(t)\alpha_2(\frac{t}{\es})\beta(\bar x) \psi_R(\frac{\bar x}{\es})\rho(\frac{-x_1}{\es})\gamma(\frac{x_2}{\es^2}))\di{x}\di{t}}_{I_{10}}\\
    +\underbrace{\es^2 D_L\int_0^T\int_{\oll}\nabla u_L^\es \cdot \nabla ( \alpha_1(t)\alpha_2(\frac{t}{\es})\beta(\bar x) \psi_L(\frac{\bar x}{\es})\rho(\frac{x_1}{\es})\gamma(\frac{x_2}{\es^2}))   \di{x}\di{t}}_{I_{11}}\\
    + \underbrace{\es^2 D_R\int_0^T\int_{\orr} \nabla u_R^\es\cdot\nabla ( \alpha_1(t)\alpha_2(\frac{t}{\es})\beta(\bar x) \psi_R(\frac{\bar x}{\es})\rho(\frac{-x_1}{\es})\gamma(\frac{x_2}{\es^2})) \di{x}\di{t}}_{I_{12}} \\
    +\underbrace{ \es\int_0^T\int_{\omm}  D_M^\es  \nabla u_M^\es\cdot\nabla (\alpha_1(t)\alpha_2(\frac{t}{\es})\beta(\bar x) \psi_M(\frac{x}{\es})\gamma(\frac{x_2}{\es^2})) \di{x}\di{t}}_{I_{13}}\\
    = \underbrace{\es^2 \int_0^T\int_{\oll} f_L(\alpha_1(t)\alpha_2(\frac{t}{\es})\beta(\bar x) \psi_L(\frac{\bar x}{\es})\rho(\frac{x_1}{\es})\gamma(\frac{x_2}{\es^2})) \di{x}\di{t}}_{I_{14}}.
\end{multline}
Similar to the case of \eqref{eq:phi1weakform}, we derive that $I_8,I_9,I_{10},I_{11},I_{12},I_{14}\rightarrow 0$ as $\es\rightarrow 0$. It remains to check the limit of the term $I_{13}$ as $\es\rightarrow 0$.  It holds that
\begin{multline*}
    I_{13}= \es\int_0^T\int_{\omm}  D_M^\es  \nabla u_M^\es\cdot (\alpha_1(t)\alpha_2(\frac{t}{\es})\nabla\beta(\bar x) \psi_M(\frac{x}{\es})\gamma(\frac{x_2}{\es^2})) \di{x}\di{t}\\
    +\int_0^T\int_{\omm}  D_M^\es  \nabla u_M^\es\cdot (\alpha_1(t)\alpha_2(\frac{t}{\es})\beta(\bar x) \nabla\psi_M(\frac{x}{\es})\gamma(\frac{x_2}{\es^2})) \di{x}\di{t}\\
    +\frac1\es\int_0^T\int_{\omm}  D_M^\es  \nabla u_M^\es\cdot (\alpha_1(t)\alpha_2(\frac{t}{\es})\beta(\bar x) \psi_M(\frac{x}{\es})\gamma'(\frac{x_2}{\es^2})e_2) \di{x}\di{t}.
\end{multline*}
By means of \Cref{def:multiscale_convergence} and \eqref{eq:multiscale_convergence_for_gradient_term}, we see that the first two terms converge to $0$. For the final term, taking the limit $\es\rightarrow 0$, we obtain
\begin{align*}
    &\!\!\!\!\lim_{\es\rightarrow 0}\frac1\es\int_0^T\int_{\omm}  D_M^\es  \nabla u_M^\es\cdot (\alpha_1(t)\alpha_2(\frac{t}{\es})\beta(\bar x) \psi_M(\frac{x}{\es})\gamma'(\frac{x_2}{\es^2})) e_2\di{x}\di{t}\\
    &=\int_0^T\int_{\Sigma}\int_\T\int_{Y}\int_Z   D_M(t,\overline{x},\tau,y,z)
    \left( \nabla_{\overline{x}}u_M^0(t,\overline{x})+\nabla_y u_M^1(t,\overline{x},\tau,y)+\nabla_zu_M^2(t,\overline{x},\tau,y,z)\right)\\
    &\qquad\qquad\qquad\qquad \cdot ( \alpha_1(t)\alpha_2(\tau)\beta(\bar x) \psi_M(y)\gamma'(z)e_2) \di{z}\di{y}\di{\tau}\di{x}\di{t}.
\end{align*}
Therefore, taking the limit $\es\rightarrow 0$ in \eqref{eq:phi2weakform}, we deduce the following identity
\begin{multline*}
    \int_0^T\int_{\Sigma}\int_\T\int_{Y}\int_Z   D_M(t,\overline{x},\tau,y,z)
    \left( \nabla_{\overline{x}}u_M^0(t,\overline{x})+\nabla_y u_M^1(t,\overline{x},\tau,y)+\nabla_zu_M^2(t,\overline{x},\tau,y,z)\right)\\
    \cdot ( \alpha_1(t)\alpha_2(\tau)\beta(\bar x) \psi_M(y)\gamma'(z)e_2) \di{z}\di{y}\di{\tau}\di{x}\di{t}=0.
\end{multline*}
This is precisely the weak form of  \eqref{eq:strong_form_for_u_2}.
\end{proof}

We now focus on deriving the initial conditions for the limit functions $u_L^0$, $u_R^0$, and $u_M^0$.

\begin{lemma}\label{lemma:initial_condition}
     Assume $(A1)-(A2)$ and \eqref{eq:energy_estimate_1_case_1}--\eqref{eq:energy_estimate_3_case_1} hold. Then,  the functions $u_L^0,u_R^0,u_M^0$ from \Cref{theorem:convergence in the bulk} and \Cref{theorem:convergence_in_the_layer} satisfy the following initial conditions
    \begin{align*}
u_L^0(0) &= v_1, &
u_M^0(0) &= 0, &
u_R^0(0) &= 0, \\
\frac{\partial u_L^0}{\partial t}(0) &= v_2, &
\frac{\partial u_M^0}{\partial t}(0) &= 0, &
\frac{\partial u_R^0}{\partial t}(0) &= 0.
\end{align*}

\end{lemma}
\begin{proof}
    We have $u_L^\es(0)=v_1,\,u_R^\es(0)=0$ and $\dfrac{\partial u_L^\es}{\partial t}(0)=v_2, \,\dfrac{\partial u_R^\es}{\partial t}(0)=0$
    with $v_1,v_2\in C_c(\oll)$. In the weak form of the microscopic problem, choose $\phi\in C_c^\infty([0,T)\times \Omega_L)$ and  $\es$ small enough such that $\operatorname{supp}( \phi)\cup \operatorname{supp}( v_1)\cup \operatorname{supp}(v_2)\subset [0,T)\times \Omega_L^\es$. Integrating the weak form from $0$ to $T$, then performing the integration by parts with respect to the time variable, we get
    \begin{multline*}
    -\rho_L\int_0^T\int_{\oll}  \dfrac{\partial u_L^\es}{\partial t} \dfrac{\partial \phi}{\partial t} \di{x}\di{t}
    +D_L\int_0^T\int_{\oll}\nabla u_L^\es \cdot \nabla \phi   \di{x}\di{t} \\+\int_{\oll}  \dfrac{\partial u_L^\es(0)}{\partial t}  \phi(0)\di{x}
    = \int_0^T\int_{\oll} f_L\phi \di{x}\di{t}.
\end{multline*}
That is
\begin{multline}\label{eq:omega_l_weakform}
    -\rho_L\int_0^T\int_{\oll}  \dfrac{\partial u_L^\es}{\partial t} \dfrac{\partial \phi}{\partial t} \di{x}\di{t}
    +D_L\int_0^T\int_{\oll}\nabla u_L^\es \cdot \nabla \phi  \di{x}\di{t} \\+\int_{\oll}  v_2  \phi(0)\di{x}
    = \int_0^T\int_{\oll} f_L\phi \di{x}\di{t}.
\end{multline}
Taking the limit $\es\rightarrow 0$, we get
\begin{align*}
    -\rho_L\int_0^T\int_{\Omega_L}  \dfrac{\partial u_L^0}{\partial t} \dfrac{\partial \phi}{\partial t} \di{x}\di{t}
    +D_L\int_0^T\int_{\Omega_L}\nabla u_L^0 \cdot \nabla \phi  \di{x}\di{t} +\rho_L\int_{\Omega_L}  v_2  \phi(0) \di{x}
    = \int_0^T\int_{\Omega_L} f_L\phi\di{x}\di{t}.
\end{align*}
Using integration by parts with respect to the time variable and \eqref{eq:macro_ul} yields to
\begin{align*}
    \rho_L\int_{\Omega_L}  v_2  \phi(0) \di{x}= \rho_L\int_{\Omega_L}\dfrac{\partial u_L^0(0)}{\partial t} \phi(0) \di{x}.
\end{align*}
That is $\dfrac{\partial u_L^0(0)}{\partial t}=v_2$. Similarly, we get  $\dfrac{\partial u_R^0(0)}{\partial t}=0$. To derive $u_L^0(0)$, we again use integration by parts with respect to the time variable on the identity \eqref{eq:omega_l_weakform}, and obtain
\begin{multline*}
    \rho_L\int_0^T\int_{\oll}  u_L^\es \dfrac{\partial^2 \phi}{\partial t^2} \di{x}\di{t}
    +D_L\int_0^T\int_{\oll}\nabla u_L^\es \cdot \nabla \phi  \di{x}\di{t} \\+\rho_L\int_{\oll}  v_2  \phi(0)\di{x}-\rho_L\int_{\oll}  v_1  \dfrac{\partial \phi(0)}{\partial t}\di{x}
    = \int_0^T\int_{\oll} f_L\phi_0 \di{x}\di{t}
\end{multline*}
taking the limit $\es\rightarrow 0$, we have
\begin{multline*}
     \rho_L\int_0^T\int_{\Omega_L}   u_L^0 \dfrac{\partial^2 \phi}{\partial t^2} \di{x}\di{t}
    +D_L\int_0^T\int_{\Omega_L}\nabla u_L^0 \cdot \nabla \phi  \di{x}\di{t} \\+\rho_L\int_{\Omega_L}  v_2  \phi(0) \di{x}-\rho_L\int_0^T\int_{\Omega_L}  u_L^0(0) \dfrac{\partial \phi(0)}{\partial t} \di{x}\di{t}
    = \int_0^T\int_{\Omega_L} f_L\phi \di{x}\di{t}.
\end{multline*}
Rewriting this in terms of the initial condition, we get
\begin{multline*}
     \rho_L\int_0^T\int_{\Omega_L}   u_L^0 \dfrac{\partial^2 \phi}{\partial t^2} \di{x}\di{t}
    +D_L\int_0^T\int_{\Omega_L}\nabla u_L^0 \cdot \nabla \phi  \di{x}\di{t} \\+\rho_L\int_{\Omega_L}  v_2  \phi(0) \di{x}-\rho_L\int_{\Omega_L}  v_1 \dfrac{\partial \phi(0)}{\partial t} \di{x}
    = \int_0^T\int_{\Omega_L} f_L\phi \di{x}\di{t}.
\end{multline*}
Now, again using integration by parts with respect to the time variable, the initial condition $\dfrac{\partial u_L^0(0)}{\partial t}=v_2$  and \eqref{eq:macro_ul}, we get
\begin{align*}
    \rho_L\int_{\Omega_L}  v_1 \dfrac{\partial \phi(0)}{\partial t} \di{x}=\rho_L\int_{\Omega_L}  u_L^0(0) \dfrac{\partial \phi(0)}{\partial t} \di{x}.
\end{align*}
Since $\phi\in C_c^\infty([0,T)\times \Omega_L)$ is arbitrary, we deduce $u_L^0(0)=v_1$. Similarly, we get $u_R^0(0)=0$. 

Let $\phi\in C_c([0,T)\times \Sigma;C_{\sharp}^\infty(Y))$, we have 
\begin{align*}
    \frac{1}{\es}\Mint\frac{\p u_M^\es}{\p t}\phi(t,\bar x,\frac{x}{\es})\di{x}\di{t}&=- \frac{1}{\es}\Mint u_M^\es\frac{\p \phi(t,\bar x,\frac{x}{\es})}{\p t}\di{x}\di{t}+ \frac{1}{\es}\int_{\omm}u_M^\es(0)\phi(0)\di{x}\\
    &=- \frac{1}{\es}\Mint u_M^\es\frac{\p \phi(t,\bar x,\frac{x}{\es})}{\p t}\di{x}\di{t}.
\end{align*}
Letting $\es\rightarrow 0$, yields
\begin{multline*}
    \mint  \left(\dfrac{\partial u_M^0(t,\bar x)}{\partial t}+ \dfrac{\partial u_M^1(t,\overline{x},\tau,y)}{\partial \tau}\right)\phi(t,\bar x,y)\dzdy\\=\mint u_M^0(t,\bar x)\frac{\p \phi(t,\bar x,y)}{\p t}\dzdy.
\end{multline*}
Applying the integration by parts with respect to $\tau$ variable and using the periodic property of  $u_M^1$ with respect to the $\tau$ variable, we arrive at
\begin{multline*}
    \mint  \dfrac{\partial u_M^0(t,x)}{\partial t}\phi(t,\bar x,y)\dzdy\\=\mint u_M^0(t,\bar x)\frac{\p \phi(t,\bar x,y)}{\p t}\dzdy.
\end{multline*}
Eventually, using the integration by parts with respect to the $t$ variable, we obtain $u_M^0(0)=0$.

Finally, to derive the initial condition for $\frac{\p u_M^0}{\p t}(0)$,  we choose $\phi=\phi_0\in C_c([0,T)\times \Omega)$
in the weak form \eqref{eq:weakform} take the limit $\es\rightarrow 0$  to get $\frac{\p u_M^0}{\p t}(0)=0$.

\end{proof}

We now state our main result concerning the effective model and the associated transmission conditions. Using the preceding lemmas, we show that the effective model satisfies the structure of a wave equation in both the bulk regions. As a first transmission condition,  we obtain the continuity of the solution along the effective interface $\Sigma$. More precisely, the traces of the bulk solutions $u_L^0$ and $u_R^0$ coincide with the interfacial limit $u_M^0$ on $\Sigma$. For the second transmission condition, we see a balance law for the jump of the normal fluxes across $\Sigma$. This jump condition is governed by a dynamical equation of $u_M^0$, involving both its time derivative and the tangential gradient. The corresponding effective quantities are characterized through suitable cell problems of elliptic and hyperbolic type.

    \begin{theorem}
   Assume $(A1)-(A3)$ and \eqref{eq:energy_estimate_1_case_1}--\eqref{eq:energy_estimate_3_case_1} hold. Let $u_L^0,u_R^0$ and $u_M^0$ be the limit functions from  \Cref{theorem:convergence in the bulk} and \Cref{theorem:convergence_in_the_layer}. Then, in the sense of distributions, the following effective model holds:
 \begin{subequations}  
\begin{align}
 \dfrac{\partial}{\partial t}\left( {\rho_L}\dfrac{\partial u_L^0}{\partial t}\right)-\mathrm{div}(D_L \nabla u_L^0)&=f_L\qquad &\mbox{in}\qquad (0,T)\times \Omega_L\label{eq:equation_for_u_L}\\
  \dfrac{\partial}{\partial t}\left({\rho_R} \dfrac{\partial u_R^0}{\partial t}\right)-\mathrm{div}(D_R \nabla u_R^0)&=0\qquad &\mbox{in}\qquad (0,T)\times\Omega_R\label{eq:equation_for_u_R}
  \end{align}
  with the effective transmission conditions
  \begin{align}
 {u_R^0=u_L^0}&{=u_M^0}\qquad &\mbox{on}\qquad(0,T)\times\Sigma\label{eq:first_transmission_condition}
\end{align}
and 
\begin{multline}\label{eq:second_transmission_condition}
     {(D_R \nabla u_R^0)\cdot n-(D_L \nabla u_L^0)\cdot n}\\
     =\frac{\p}{\p t}\left(\rho_{1}^*\dfrac{\partial u_M^0}{\partial t}+\rho_{2}^*\cdot \nabla_{\bar x}u_M^0\right)
        -\mathrm{div_{\bar x}}\left(D_1^* \nabla_{\bar x}u_M^0+D_2^*\dfrac{\partial u_M^0}{\partial t}\right)\\
    \qquad \mbox{on}\qquad(0,T)\times\Sigma.
\end{multline}

The effective coefficients are,
  \begin{align*}
        \rho_{1}^*:=\int_\T\int_{Y}\int_Z \rho_M\left(1+\frac{\p h}{\p \tau}\right)\di{z}\di{y}\di{\tau},
    \end{align*}
    \begin{align*}
        {\rho_{2}^*}:=\int_\T\int_{Y}\int_Z\rho_M \begin{bmatrix}
            \frac{\p \psi_1}{\p \tau}\\
            \frac{\p \psi_2}{\p \tau}
        \end{bmatrix}\di{z}\di{y}\di{\tau}
    \end{align*}
    \begin{align*}
        D_1^*:= \int_0^1\int_Y\int_Z D_M\left(I+\begin{bmatrix}
0 & 0 \\
0 & \frac{d w}{dz}
\end{bmatrix}\right)\left(I+\begin{bmatrix}
\frac{\partial \psi_1}{\partial y_1} & \frac{\partial \psi_2}{\partial y_1} \\
\frac{\partial \psi_1}{\partial y_2} & \frac{\partial \psi_2}{\partial y_2}
\end{bmatrix}\right) \di{z}\di{y}\di{\tau},
    \end{align*}
     \begin{align*}
        D_2^*:=\int_\T\int_{Y}\int_Z D_M\left(I+\begin{bmatrix}
0 & 0 \\
0 & \frac{d w}{dz}
\end{bmatrix}\right)\nabla_y h\di{z}\di{y}\di{\tau}
    \end{align*}

      where $w(t,\overline{x},\tau,y,\cdot)\in H^1_{\sharp,0}(Z)$ is the unique weak solution of the $Z$ cell problem
   \begin{align}\label{eq:zcell_problem}
        \frac{d}{dz}\left(D_M\left(1+\frac{dw}{dz}\right)\right)=0\quad\mbox{in}\quad Z,\\
        \qquad 
         w \mbox{ is z-periodic,} \qquad \int_Z w\di{z}=0,
    \end{align}
   
    $h(t,\bar x, \cdot)\in \mathcal{H}_{\sharp,0}(\T,Y)$ is the unique weak solution of the cell problem 
    \begin{align}\label{eq:cellproblemh}
        \frac{\partial}{\partial \tau}\left(\bar \rho_M \frac{\partial h}{\partial \tau}\right)-\mathrm{div}_y\left(\bar D_M\nabla_yh\right)&=-\frac{\partial \bar \rho_M}{\partial \tau}\quad\mbox{in}\quad \T\times Y\\
        \left(\bar D_M\nabla_yh\right)\cdot \nu &=0 \quad\mbox{on}\quad Y_L\cup Y_R,\\
         h \mbox{ is } (\tau,y_2)\mbox{-periodic,}&      \int_\T\int_Yh\di{y}\di{\tau}=0
    \end{align}
and $\psi_1(t,\bar x, \cdot),\psi_2(t,\bar x,\cdot)\in \mathcal{H}_{\mathcal{L}}$ are the unique weak solution of
    \begin{align}\label{eq:cellproblem_psi}
        \frac{\partial}{\partial \tau}\left(\bar \rho_M \frac{\partial \psi_i}{\partial \tau}\right)-\mathrm{div}_y\left(\bar D_M\nabla_y\psi_i\right)&=\mathrm{div}_y\left(\bar D_Me_i\right)\quad\mbox{in}\quad \T\times Y\\
      \left(\bar D_M\nabla_y \psi_i\right)\cdot \nu &=-\left(\bar D_Me_i\right)\cdot \nu \quad\mbox{on}\quad Y_L\cup Y_R,\\
       \psi_i \mbox{ is } (\tau,y_2)\mbox{-periodic,}&\quad  \int_\T\int_Y\psi_i\di{y}\di{\tau}=0
    \end{align}
    and $\bar \rho_M:=\int_Z\rho_M\di{z}$ and $\bar D_M=\begin{bmatrix}
            \int_Z D_M\di{z}&0\\
            0&  \left(\int_Z\frac{1}{ D_M}\di{z}\right)^{-1}\end{bmatrix}$.
    Further, we have the following initial conditions
   \begin{align*}
u_L^0(0) &= v_1, &
u_M^0(0) &= 0, &
u_R^0(0) &= 0, \\
\frac{\partial u_L^0}{\partial t}(0) &= v_2, &
\frac{\partial u_M^0}{\partial t}(0) &= 0, &
\frac{\partial u_R^0}{\partial t}(0) &= 0.
\end{align*}
    \end{subequations}
  
\end{theorem}


\begin{proof} 
   From \Cref{lemma:derivation_of_uL_n_uR} and \Cref{lemma:first_transmission_condition} we get \eqref{eq:equation_for_u_L}--\eqref{eq:first_transmission_condition}. To show the second transmission condition, we first write $u_M^1$ and $u_M^2$ in terms of $u_M^0$ and solutions of suitable cell problems.
 The structure of the equation \eqref{eq:strong_form_for_u_2} allows us to express $u_M^2$ in terms of $ u_M^0$ and $ u_M^1$, as well as the solution of a cell problem. Precisely, we can write
 
\begin{align*}   u_M^2(t,\overline{x},\tau,y,z)=\Big(\nabla_{\overline{x}}u_M^0(t,\overline{x})+\nabla_y u_M^1(t,\overline{x},\tau,y)\Big)\cdot W(t,\overline{x},\tau,y,z),
\end{align*}
where $W:=(0,w)$ and $w$ is the weak solution to the cell problem \eqref{eq:zcell_problem}.

   Now, we aim to write $u_M^1$ in terms of $u_M^0$ and suitable cell problems in $\T\times Y$. Let 
 \begin{align}\label{eq:DMbar_structuer2}
        \bar D_M(t,\overline{x},\tau,y):=\begin{bmatrix}
            \int_Z D_M(t,\overline{x},\tau,y,z)\di{z}&0\\
            0&\int_Z D_M(t,\overline{x},\tau,y,z)\left(1+\frac{dw(t,\overline{x},\tau,y,z)}{dz}\right)\di{z}
        \end{bmatrix}
    \end{align}
    
     and
     \begin{align*}
         \bar \rho_M(t,\overline{x},\tau,y):=\int_Z \rho_M(t,\overline{x},\tau,y,z) \di{z}.
     \end{align*}
     Then we can write \eqref{eq:strong_form_u_1} in the following simplified form:
     \begin{align}\label{eq:equation_for_u1}
          \frac{\partial}{\partial \tau}\left(\bar \rho_M \frac{\partial u_M^1}{\partial \tau}\right)-\mathrm{div}_y\left(\bar D_M\nabla_yu_M^1\right)&=-\frac{\partial \bar \rho_M}{\partial \tau}\frac{\p u_M^0}{\p t}+\mathrm{div}_y\left(\bar D_M\nabla_{\bar x}u_M^0\right)\quad\mbox{in}\quad (0,T)\times\Sigma\times \T\times Y\\
          \left(\bar D_M\nabla_yu_M^1\right)\cdot n &=-\left(\bar D_M\nabla_{\bar x}u_M^0\right)\cdot n \qquad\mbox{on } Y_L\cup Y_R
     \end{align}
Observe that, from  \eqref{eq:zcell_problem}, we have 
\begin{align*}
    D_M\left(1+\frac{dw}{dz}\right)=p
\end{align*}
for some $p$ independent of the variable $z$. Since $w$ is a $ Z$-periodic function, a simple calculation gives
\begin{align*}
    p=\frac{1}{\int_Z\frac{1}{D_M}\di{z}},
\end{align*}
    and  therefore
     \begin{align*}
         [\bar D_M]_{22}=\left(\int_Z\frac{1}{ D_M}\di{z}\right)^{-1}.
     \end{align*}
   Due to the presence of temporal and spatial derivatives of $u_M^0$ in \eqref{eq:equation_for_u1}, we cannot directly construct the cell problem for $u_M^1$. Since \eqref{eq:equation_for_u1} is a linear problem, we split the equation into two.  Consider the weak solution $v_1$  for the problem
     \begin{align}
          \frac{\partial}{\partial \tau}\left(\bar \rho_M \frac{\partial v_1}{\partial \tau}\right)-\mathrm{div}_y\left(\bar D_M\nabla_yv_1\right)&=-\frac{\partial \bar \rho_M}{\partial \tau}\frac{\p u_M^0}{\p t}\quad&\mbox{in}\quad \T\times Y
          \\
          \left(\bar D_M\nabla_yv_1\right)\cdot n &=0 \qquad&\mbox{on }\quad Y_L\cup Y_R
     \end{align}
   together with the periodic boundary with respect to $\tau$ and $y_2$, and the zero mean condition $\int_0^1\int_{Y}v_1\di{y}\di{\tau=0}$ and the weak solution $v_2$ of the problem
     \begin{align}
          \frac{\partial}{\partial \tau}\left(\bar \rho_M \frac{\partial v_2}{\partial \tau}\right)-\mathrm{div}_y\left(\bar D_M\nabla_yv_2\right)&=\mathrm{div}_y\left(\bar D_M\nabla_{\bar x}u_M^0\right)\qquad&\mbox{in}\quad \T\times Y\\
          \left(\bar D_M\nabla_yv_2\right)\cdot n &=-\left(\bar D_M\nabla_{\bar x}u_M^0\right)\cdot n \qquad&\mbox{on } Y_L\cup Y_R
     \end{align}
     together with the periodic boundary with respect to $\tau$ and $y_2$, and $\int_0^1\int_{Y}v_2\di{y}\di{\tau=0}$.   
Then it follows that, $u_M^1=v_1+v_2$.


     For any fixed $(t,\bar x)\in (0,T)\times\Sigma$, let 
     $h(t,\bar x, \tau, y)$ be the weak solution of the cell problem \eqref{eq:cellproblemh}
and $\psi_1(t,\bar x, \tau, y),\psi_2(t,\bar x, \tau, y)$ are the weak solutions of the problem \eqref{eq:cellproblem_psi}. Then $v_1=\frac{\partial u_M^0}{\partial t}h$ and $v_2:=\nabla_{\bar x}u_M^0\cdot \Psi$.
    
    The  compatibility condition $(A3)$ guarantees the existence and uniqueness of both $h$ and $\Psi$. Therefore, we get 
    \begin{align*}
         u_M^1&=\frac{\partial u_M^0}{\partial t}h+\nabla_{\bar x}u_M^0\cdot \Psi.
    \end{align*}
Using the aforementioned structure, we have
    \begin{align*}
        \int_\T\int_{Y}\int_Z\left(\rho_M\left(\dfrac{\partial u_M^0}{\partial t}+ \dfrac{\partial u_M^1}{\partial \tau}\right)\right)\di{z}\di{y}\di{\tau}=\rho_{1}^*\dfrac{\partial u_M^0}{\partial t}+\rho_{2}^*\cdot \nabla_{\bar x}u_M^0,
    \end{align*}

    where
    \begin{align*}
        \rho_{1}^*:=\int_{\T}\int_{Y}\bar \rho_M\left(1+\frac{\p h}{\p \tau}\right)\di{y}\di{\tau} \qquad   {\rho_{2}^*}\cdot e_i:=\int_{\T}\int_{Y}\bar \rho_M \frac{\p \psi_i}{\p \tau}\di{y}\di{\tau},
    \end{align*}
   
and
   
    \begin{align*}
        \int_\T\int_{Y}\int_ZD_M \left( \nabla_{\overline{x}}u_M^0+\nabla_y u_M^1+\nabla_zu_M^2\right)\di{z}\di{y}\di{\tau}&= \int_{\T}\int_{Y}\bar D_M \left( \nabla_{\overline{x}}u_M^0+\nabla_y u_M^1\right)\di{y}\di{\tau}\\
        &=D_1^* \nabla_{\bar x}u_M^0+D_2^*\dfrac{\partial u_M^0}{\partial t}
    \end{align*}

    where
    \begin{align*}
        D_2^*:=\int_\T\int_{Y}\int_Z D_M\begin{bmatrix}
            1&0\\
            0&1+\frac{dw}{dz}
        \end{bmatrix}\nabla_y h\di{z}\di{y}\di{\tau}
    \end{align*}
    and 
    \begin{align*}
        D_1^*:=\int_\T\int_{Y}\int_Z D_M \begin{bmatrix}
            1&0\\
            0&1+\frac{dw}{dz}
        \end{bmatrix}
       \left( \begin{bmatrix}
            1&0\\
            0&1
        \end{bmatrix}+\nabla_y\Psi\right)\di{z}\di{y}\di{\tau}.
    \end{align*}
    Therefore, the transmission condition \eqref{eq:transmition_condition_2}, can be written as 
    \begin{multline*}
        {(D_R \nabla u_R^0)\cdot n-(D_L \nabla u_L^0)\cdot n}=\frac{\p}{\p t}\left(\rho_{1}^*\dfrac{\partial u_M^0}{\partial t}+\rho_{2}^*\cdot \nabla_{\bar x}u_M^0\right)
        -\mathrm{div_{\bar x}}\left(D_2^*\dfrac{\partial u_M^0}{\partial t}+D_1^* \nabla_{\bar x}u_M^0\right).
    \end{multline*}
The initial conditions follow from \Cref{lemma:initial_condition}.
\end{proof}
\subsection{Uniqueness of the effective model}
In the aforementioned theorem, we established the existence of the effective model. We now look for uniqueness of the model. Even though the effective model is formulated as a linear problem, the presence of nonstandard transmission conditions makes the uniqueness not so obvious. Here, under some additional assumptions on the effective coefficients, we prove the uniqueness of the upscaled model.  For that, we first prove a lemma that shows the symmetric structure of the effective coefficient $D_1^*$.
\begin{lemma}\label{lemma:D1symmetric}
    The effective matrix $D_1^*$ is symmetric.
\end{lemma}
\begin{proof}
   Using the structure of $\bar D_M$ \eqref{eq:DMbar_structuer2}, and the definition of $D_1^*$, we get
    \begin{align*}
        D_1^*=\int_\T\int_{Y}\bar D_M 
       \left( \begin{bmatrix}
            1&0\\
            0&1
        \end{bmatrix}+\nabla_y\Psi\right)\\di{y}\di{\tau}.
    \end{align*}
    Therefore,
    \begin{align*}
        [D_1^*]_{ij}&=\int_\T\int_{Y}e_i\cdot \bar D_M (e_j+\nabla_y \psi_j)\di{y}\di{\tau}\\
        &=\int_\T\int_{Y}e_i\cdot \bar D_M e_j\di{y}\di{\tau}+\int_\T\int_{Y}e_i\cdot \bar D_M\nabla_y \psi_j\di{y}\di{\tau}.
    \end{align*}
    Since $\bar D_M $ is symmetric, the first term on the right-hand side is symmetric. It remains to show that the second term is symmetric.
    Consider the weak formulation of the problem \eqref{eq:cellproblem_psi} for $\psi_i$, and use the test function $\psi_j$, we obtain
    \begin{align*}
        \int_\T\int_Y\left(\bar De_i\right)\nabla_y\psi_j\di{y}\di{\tau}=\int_\T\int_Y \left(\bar \rho \frac{\partial \psi_i}{\partial \tau}\right)\frac{\partial \psi_j}{\partial \tau}\di{y}\di{\tau}-\int_\T\int_Y\left(\bar D\nabla_y\psi_i\right)\nabla_y\psi_j\di{y}\di{\tau}.
    \end{align*}
    Similarly, taking the weak formulation for $\psi_j$ and testing with $\psi_i$ gives
     \begin{align*}
        \int_\T\int_Y\left(\bar De_j\right)\nabla_y\psi_i\di{y}\di{\tau}=\int_\T\int_Y \left(\bar \rho \frac{\partial \psi_j}{\partial \tau}\right)\frac{\partial \psi_i}{\partial \tau}\di{y}\di{\tau}-\int_\T\int_Y\left(\bar D\nabla_y\psi_j\right)\nabla_y\psi_i\di{y}\di{\tau}.
    \end{align*}
   From the aforementioned identities and since $\bar D_M$ is symmetric, we have
    \begin{align*}
        \int_\T\int_Y\left(\bar De_i\right)\nabla_y\psi_j\di{y}\di{\tau}= \int_\T\int_Y\left(\bar De_j\right)\nabla_y\psi_i\di{y}\di{\tau}.
    \end{align*}
    Therefore we obtain $[D_1^*]_{ij}=[D_1^*]_{ji}$.
\end{proof}

    Next, under suitable additional structural assumptions on the effective coefficients, we establish the uniqueness of the effective model.
\begin{theorem}
     Assume $(A1)-(A3)$ and \eqref{eq:energy_estimate_1_case_1}--\eqref{eq:energy_estimate_3_case_1} hold. Assume further that $\rho_1^*,\rho_2^*,D_1^*,D_2^*
\in W^{1,\infty}\bigl((0,T)\times\Sigma\bigr)$ and $\rho_1^*$ is strictly positive and $D_1^*$ is uniformly positive definite. Then the  effective model admits a unique weak solution.
\end{theorem}
\begin{proof}
    We show the uniqueness of the effective model by assuming there are two solutions to the effective model, and then showing that these solutions coincide. 

    Let $(u_L,u_M,u_R)$ and $(v_L,v_M,v_R)$ be two weak solutions to the effective model \eqref{eq:equation_for_u_L}--\eqref{eq:second_transmission_condition}. Set $(w_L,w_M,w_R):=(u_L,u_M,u_R)-(v_L,v_M,v_R)$, taking the weak formulation of both solution sets with the same test function and subtracting the two weak forms gives
    \begin{multline}\label{eq:weak_form_of_difference}
        -\rho_L\int_0^T\int_{\mol}\frac{\p w_L}{\p t}\frac{\p \phi_L}{\p t}\di{x}\di{t}+D_L\int_0^T\int_{\mol}\nabla w_L\nabla \phi_L\di{x}\di{t}-\rho_R\int_0^T\int_{\mor}\frac{\p w_R}{\p t}\frac{\p \phi_R}{\p t}\di{x}\di{t}\\
        +D_R\int_0^T\int_{\mor}\nabla w_R\nabla \phi_R\di{x}\di{t}
        -\int_0^T\int_\Sigma \left(\rho_{1}^*\dfrac{\partial w_M}{\partial t}+\rho_{2}^*\cdot \nabla_{\bar x}\frac{\p w_M}{\p t}\right)\frac{\p w_M}{\p t}\di{\bar x}\di{t}
        \\ +\int_0^T\int_{\Sigma}\left(D_1^* \nabla_{\bar x}w_M+D_2^*\dfrac{\partial w_M}{\partial t}\right)\nabla_{\bar x}\phi_M\di{\bar x}\di{t}=0
    \end{multline} 
    for all $\phi\in L^2((0,T);H^1_{\sharp}(\Omega))\cap H^1((0,T);L^2(\Omega))$ with $\Phi(T)=0$.
Let $r\in (0,T]$, we define 
\begin{align*}
    V^r (t):=\begin{cases}
        \int_t^r w(s)\di{s} \qquad &t\in [0,r]\\
        0\qquad &r\leq t\leq T.
    \end{cases}
\end{align*}
Then clearly, $V^r\in L^2((0,T);H^1_{\sharp}(\Omega))\cap H^1((0,T);L^2(\Omega))$ and $\frac{\p V^r}{\p t}=-w$. Choosing $\phi=V^r$ in \eqref{eq:weak_form_of_difference}, we arrive at
\begin{multline}\label{eq:testing_with_Vr}
      \rho_L\int_0^r\int_{\mol}\frac{\p w_L}{\p t}w_L\di{x}\di{t}+D_L\int_0^r\int_{\mol}\nabla w_L\nabla V^r_L\di{x}\di{t}+\rho_R\int_0^r\int_{\mor}\frac{\p w_R}{\p t}w_R\di{x}\di{t}\\
        +D_R\int_0^r\int_{\mor}\nabla w_R\nabla V^r_R\di{x}\di{t}
        +\int_0^r\int_\Sigma \left(\rho_{1}^*\dfrac{\partial w_M^0}{\partial t}+\rho_{2}^*\cdot \nabla_{\bar x}w_M\right)w_M\di{\bar x}\di{t}
        \\
        +\int_0^r\int_{\Sigma}\left(D_1^* \nabla_{\bar x}w_M+D_2^*\dfrac{\partial w_M}{\partial t}\right)\nabla_{\bar x}V^r_M\di{\bar x}\di{t}=0.
\end{multline}
Since $w(0)=0,$ by integration by parts with respect to the time variable, we obtain the following identities
\begin{align*}
    \rho_L\int_0^r\int_{\mol}\frac{\p w_L}{\p t}w_L\di{x}\di{t}=\frac{\rho_L}{2}\|w_L(r)\|_{L^2(\mol)}^2, \qquad \rho_R\int_0^r\int_{\mor}\frac{\p w_R}{\p t}w_R\di{x}\di{t}=\frac{\rho_R}{2}\|w_R(r)\|_{L^2(\mor)}^2,
    \end{align*}
    and
    \begin{align*}
    D_L\int_0^r\int_{\mol}\nabla w_L\nabla V^r_L\di{x}\di{t}=\frac{D_L}{2}\|\nabla V^r_L(0)\|_{L^2(\mol)}^2 \\
    D_R\int_0^r\int_{\mol}\nabla w_L\nabla V^r_R\di{x}\di{t}=\frac{D_R}{2}\|\nabla V^r_R(0)\|_{L^2(\mor)}^2.
\end{align*}
For the interface terms, using integration by parts with respect to the space variable yields
\begin{align*}
    \int_0^r\int_\Sigma \rho_{1}^*\dfrac{\partial w_M^0}{\partial t} w_M\di{\bar x}\di{t}&=\frac{1}{2}\int_\Sigma \rho_{1}^*(r)|w_M(r)|^2\di{\bar x}-\frac12 \int_0^r\int_\Sigma \frac{\p \rho_{1}^*}{\p t} |w_M|^2\di{\bar x}\di{t}\\
    \int_0^r\int_\Sigma \rho_{2}^*\cdot \nabla_{\bar x}w_M w_M\di{\bar x}\di{t}&=-\frac{1}{2}\int_0^r \int_{\Sigma} \nabla_{\bar x}\cdot \rho_{2}^*|w_M|^2\di{\bar x}\di{t}.
\end{align*}
By \Cref{lemma:D1symmetric}, we have $D_1^*$ is symmetric,  using $\nabla_{\bar x} w_M=-\frac{\p }{\p t}(\nabla_{\bar x}V_M^r)$, we get 
\begin{multline*}
    \int_0^r\int_{\Sigma}D_1^* \nabla_{\bar x}w_M\nabla_{\bar x}V^r_M\di{\bar x}\di{t}=\frac{1}{2}\int_\Sigma D_1^*(0)\nabla_{\bar x}V_M^r(0)\cdot \nabla_{\bar x}V_M^r(0)\di{\bar x}\\
    +\frac12  \int_0^r\int_{\Sigma}\frac{\p D_1^*}{\p t} \nabla_{\bar x}V_M^r\cdot \nabla_{\bar x}V_M^r\di{\bar x}\di{t},
\end{multline*}
and
\begin{align*}
    \int_0^r\int_{\Sigma}D_2^*\dfrac{\partial w_M}{\partial t}\nabla_{\bar x}V^r_M\di{\bar x}\di{t}=-\int_0^r\int_{\Sigma}\frac{\p D_2^*}{\p t}w_M \cdot \nabla_{\bar x}V^r_M \di{\bar x}\di{t}-\frac12 \int_0^r\int_{\Sigma}\nabla_{\bar x}\cdot D_2^*|w_M|^2\di{\bar x}\di{t}.
\end{align*}
Now setting
\begin{align*}
    \Theta_L(t):=\int_0^t\nabla w_L (s) \di{s},\qquad  \Theta_M(t):=\int_0^t\nabla_{\bar x} w_M (s) \di{s},\qquad  \Theta_R(t):=\int_0^t\nabla w_R (s) \di{s},
\end{align*}
we have
\begin{align*}
    \Theta_L(r)=\nabla V^r_L(0), \qquad  \Theta_M(r)=\nabla_{\bar x}V^r_M(0), \qquad  \Theta_R(r)=\nabla V^r_R(0).
\end{align*}
Let 
\begin{multline*}
    E(r):=\frac12\Bigg[\rho_L\|w_L(r)\|_{L^2(\mol)}^2+\rho_R\|w_R(r)\|_{L^2(\mor)}^2+ D_L\| \Theta_L(r)\|_{L^2(\mol)}^2+D_R\|\Theta_R(r)\|_{L^2(\mor)}^2\\
    +\int_\Sigma \rho_{1}^*(r)|w_M(r)|^2\di{\bar x}+\int_\Sigma D_1^*(0)\Theta_M(r)\cdot \Theta_M(r)\di{\bar x}\Bigg]
\end{multline*}
Using \eqref{eq:testing_with_Vr} and the preceding identities, we get
\begin{multline*}
    E(r)=\frac12\Bigg[ \int_0^r\int_\Sigma \frac{\p \rho_{1}^*}{\p t} |w_M|^2\di{\bar x}\di{t}+\int_0^r \int_{\Sigma} \nabla_{\bar x}\cdot \rho_{2}^*|w_M|^2\di{\bar x}\di{t}- \int_0^r\int_{\Sigma}\frac{\p D_1^*}{\p t} \nabla_{\bar x}V_M^r\cdot \nabla_{\bar x}V_M^r\di{\bar x}\di{t}\\
    +2\int_0^r\int_{\Sigma}\frac{\p D_2^*}{\p t}w_M \cdot \nabla_{\bar x}V^r_M \di{\bar x}\di{t}+ \int_0^r\int_{\Sigma}\nabla_{\bar x}\cdot D_2^*|w_M|^2\di{\bar x}\di{t}\Bigg].
\end{multline*}
Applying the boundedness of the coefficients and Young's inequality yields
\begin{align*}
    E(r)\leq C \int_0^r\Big(\|w_M(t)\|_{L^2(\Sigma)}^2+\|\nabla_{\bar x}V_M^r(t)\|_{L^2(\Sigma)}^2\Big)\di{t}.
\end{align*}
Since we have $\nabla_{\bar x}V_M^r(t)=\Theta_M(r)-\Theta_M(t)$, we get
\begin{align*}
    \|\nabla_{\bar x}V_M^r(t)\|_{L^2(\Sigma)}^2\leq 2\|\Theta_M(r)\|_{L^2(\Sigma)}^2+2\|\Theta_M(t)\|_{L^2(\Sigma)}^2,
\end{align*}
and 
\begin{align*}
    E(r)\leq C \int_0^r\Big(\|w_M(t)\|_{L^2(\Sigma)}^2+\|\Theta_M(t)\|_{L^2(\Sigma)}^2\Big)\di{t}+Cr\|\Theta_M(r)\|_{L^2(\Sigma)}^2
\end{align*}
From the ellipticity property of $D_1^*$ and the uniform positivity of $\rho_1^*$, and from the definition of $E(\cdot)$, we obtain
\begin{align*}
    E(t)\geq C\Big(\|w_M(t)\|_{L^2(\Sigma)}^2+\|\Theta_M(t)\|_{L^2(\Sigma)}^2\Big),
\end{align*}
for some constant $C>0.$

Therefore we get
\begin{align*}
    E(r)&\leq C\int_0^rE(t)\di{t}+Cr\|\Theta_M(r)\|_{L^2(\Sigma)}^2\\
    &\leq C\int_0^rE(t)\di{t}+Cr\, E(r).
\end{align*}
Choose $r$ small enough such that $Cr\leq\frac12$, then we get
\begin{align*}
    E(r)&\leq C\int_0^rE(t)\di{t}.
\end{align*}
Now, applying the Gronwall's inequality, we deduce $E(t)=0$ in $[0,r].$ Therefore, we conclude $w=0$ on $[0,r].$ Since $w(r)=0,$ the same argument can be applied to $[r,2r], [2r,3r]$ and so on. After finitely many steps, these intervals cover $[0,T]$, and hence $w=0$ on $[0,T]$.
\end{proof}

We remark that, since the solution of the effective model is unique, the convergence results presented in \Cref{theorem:convergence in the bulk,theorem:convergence_in_the_layer,theorem: convergence in the layer left} hold for the whole sequence.

	\appendix
	\section{Appendix}
    In this Appendix, we provide the proofs of \Cref{theorem:compactness_result} and \Cref{lemma:first_transmission_condition}, along with two auxiliary lemmas. These results follow arguments similar to \cite[Proposition 4.2]{neuss2007effective} and \cite[Theorem 5.3]{effectiveapratim}, with minor modifications to accommodate the multiscale setting and the effects of dimension reduction.  For the reader’s convenience, we provide the proof here.

We first state a multi-scale oscillation lemma for thin layer and later use this to establish \Cref{theorem:compactness_result}. 
\begin{lemma}[Multi-scale oscillation lemma for thin layer]\label{lemma:oscillation_lemma}
	Let \\$\phi\left(t,\overline{x},\dfrac{t}{\es}, \dfrac{x}{\es},\dfrac{x_2}{\es^2}\right)\in L^2((0,T)\times \Sigma ;\mathcal{C}_{\sharp}(\overline{\T\times Y\times Z}))$ then we have
	\begin{itemize}
	    \item [(i)]\begin{align*}
	       \left| \dfrac{1}{\es}\int_0^T\int_{\omm}\phi\left(t,\overline{x},\dfrac{t}{\es}, \dfrac{x}{\es},\dfrac{x_2}{\es^2}\right)\di{x}\di{t}\right|\leq  \int_0^T\int_{\Sigma}\sup_{(\tau,y,z)\in {\T\times Y\times Z}} |\phi(t,\overline{x},\tau,y,z)| \di{\overline{x}}\di{t},
	    \end{align*}
	    \item[(ii)]
	    \begin{align*}
	        \lim_{\es\rightarrow 0}  \dfrac{1}{\es}\int_0^T\int_{\omm}\phi\left(t,\overline{x},\dfrac{t}{\es}, \dfrac{x}{\es},\dfrac{x_2}{\es^2}\right)\di{x}\di{t}= \mint\phi(t,\overline{x},\tau,y,z) \di{z}\di{y}\di{\tau}\di{\overline{x}}\di{t}.
	    \end{align*}
	\end{itemize}

	\end{lemma}
	\begin{proof}
	    
	\begin{itemize}
	    \item [(i)] Since $\phi(t,x,\cdot,\cdot,\cdot)$ is periodic and continuous, we have
	    \begin{align*}
	         \dfrac{1}{\es}\left|\int_0^T\int_{\omm}\phi\left(t,\overline{x},\dfrac{t}{\es}, \dfrac{x}{\es},\dfrac{x_2}{\es^2}\right)\di{x}\di{t}\right|&\leq \dfrac{1}{\es}\es \int_0^T\int_{\Sigma}\sup_{(\tau,y,z)\in {\T\times Y\times Z}} |\phi(t,\overline{x},\tau,y,z)| \di{\overline{x}}\di{t}\\
	         &=\int_0^T\int_{\Sigma}\sup_{(\tau,y,z)\in {\T\times Y\times Z}} |\phi(t,\overline{x},\tau,y,z)| \di{\overline{x}}\di{t}.
	    \end{align*}
	    Note that to ensure the above inequality, we used the fact $\phi\left(t,\overline{x},\dfrac{t}{\es}, \dfrac{x}{\es},\dfrac{x_2}{\es^2}\right)$ is measurable, this measurability follows immediately from Lemma 3.2.2 of \cite{bhattacharya2023homogenization}.
	    \item[(ii)] We have 
	    \begin{align*}
	         \dfrac{1}{\es}\int_0^T\int_{\omm}\phi\left(t,\overline{x},\dfrac{t}{\es}, \dfrac{x}{\es},\dfrac{x_2}{\es^2}\right)\di{x}\di{t}&= \dfrac{1}{\es}\int_0^T\int_{-\frac{\es}{2}}^{\frac{\es}{2}}\int_{0}^h\phi\left(t,\overline{x},\dfrac{t}{\es}, \dfrac{x_1}{\es},\dfrac{x_2}{\es},\dfrac{x_2}{\es^2}\right)\di{x_1}\di{x_2}\di{t}.
	    \end{align*}
	    Using the change of variable $x_1=\es y_1$, we obtain
	        \begin{align*}
	         \dfrac{1}{\es}\int_0^T\int_{\omm}\phi\left(t,\overline{x},\dfrac{t}{\es}, \dfrac{x}{\es},\dfrac{x_2}{\es^2}\right)\di{x}\di{t}&= \int_0^T\int_{-\frac{1}{2}}^{\frac{1}{2}}\int_{0}^h\phi\left(t,\overline{x},\dfrac{t}{\es}, {y_1},\dfrac{x_2}{\es},\dfrac{x_2}{\es^2}\right)\di{y_1}\di{x_2}\di{t}
	    \end{align*}
	    Passing $\es\rightarrow 0$ and using the oscillation Lemma (see \cite[Theorem 2.6]{cioranescu1999introduction}), we get
	    \begin{multline*}
	        \lim_{\es\rightarrow 0} \dfrac{1}{\es}\int_0^T\int_{\omm}\phi\left(t,\overline{x},\dfrac{t}{\es}, \dfrac{x}{\es},\dfrac{x_2}{\es^2}\right)\di{x}\di{t}\\= \int_0^T\int_{-\frac{1}{2}}^{\frac{1}{2}}\int_{0}^h\int_\T\int_0^1\int_Z\phi\left(t,\overline{x},\tau, {y_1},y_2,z\right)\di{z}\di{y_1}\di{y_2}\di{\tau}\di{x_2}\di{t}.
	    \end{multline*}
	    By rearranging the integrals, we have
	    \begin{multline*}
	        \lim_{\es\rightarrow 0} \dfrac{1}{\es}\int_0^T\int_{\omm}\phi\left(t,\overline{x},\dfrac{t}{\es}, \dfrac{x}{\es},\dfrac{x_2}{\es^2}\right)\di{x}\di{t}\\= \mint \phi(t,\overline{x},\tau,y,{z}) \di{z}\di{y}\di{\tau}\di{\overline{x}}\di{t}.
	    \end{multline*}
	\end{itemize}
	\end{proof}

We now give the proof of \Cref{theorem:compactness_result}. The proof follows in similar lines to \cite[Proposition 4.2]{neuss2007effective}. The only change is to extend the proof to higher-order oscillations.
    \begin{proof}[Proof of \Cref{theorem:compactness_result}]

	    Let $F_{\es}: C([0,T]\times\overline{\Sigma};C_{\sharp}( \overline{\T\times Y\times Z}))\rightarrow \mathbb{R}$ defined as
	    \begin{align*}
	        F_\es (\phi):=\dfrac{1}{\es} \int_0^T\int_{\omm} v^\es (t,x)\phi\left(t,\overline{x},\dfrac{t}{\es}, \dfrac{x}{\es},\dfrac{x_2}{\es^2}\right)\di{x}\di{t}
	    \end{align*}
	    By Cauchy–Schwarz inequality, \Cref{lemma:oscillation_lemma} and \eqref{eq:l2bound_on_v_epsilon}, we have 
	    \begin{align*}
	        |F_\es (\phi)|&\leq\frac{1}{\sqrt{\es}}\|v^\es\|_{L^2(\omm)}\left(\frac{1}{\es}\int_0^T\int_{\omm}\phi\left(t,\overline{x},\dfrac{t}{\es}, \dfrac{x}{\es},\dfrac{x_2}{\es^2}\right)\di{x}\di{t}\right)^\frac12 \\
            &\leq C \|\phi\|_{C([0,T]\times\overline{\Sigma};C_{\sharp}( \overline{\T\times Y\times Z}))}.
	    \end{align*}
	 Now using Banach–Alaoglu theorem \cite[Theorem 3.15]{rudin1991functional}, we get a subsequence of $F_{\es}$ (again denoted by $F_{\es}$) such that $F_{\es}\overset{\ast}{\rightharpoonup} F_0$, where $F_0$ is an element of dual space of $C([0,T]\times\overline{\Sigma};C_{\sharp}( \overline{\T\times Y\times Z}))$. 
	 
	 \begin{align*}
	     F_{0}(\phi)&=\lim_{\es\rightarrow 0} F_{\es}(\phi)\\
	     &\leq C  \lim_{\es\rightarrow 0}  \dfrac{1}{\es}\int_0^T\int_{\omm}\phi\left(t,\overline{x},\dfrac{t}{\es}, \dfrac{x}{\es},\dfrac{x_2}{\es^2}\right)\di{x}\di{t}\\
	     &=\mint \phi(t,\overline{x},\tau,y,z) \dzdy.
	 \end{align*}
	 	Since $C([0,T]\times\overline{\Sigma};C_{\sharp}( \overline{\T\times Y\times Z}))$ is dense in $L^2((0,T)\times \Sigma\times {\T\times Y\times Z})$ (see \cite[Lemma 4.3]{neuss2007effective}),  for any $\phi\in L^2((0,T)\times \Sigma\times {\T\times Y\times Z})$ there exists a $v_0\in L^2((0,T)\times \Sigma\times \T\times Y\times Z)$ such that
	 	\begin{align*}
	 	    \lim_{\es\rightarrow 0} \dfrac{1}{\es} \int_0^T\int_{\omm} v^\es &(t,x)\phi\left(t,\overline{x},\dfrac{t}{\es}, \dfrac{x}{\es},\dfrac{x_2}{\es^2}\right)\di{x}\di{t}= F_0(\phi)\\
	 	    &=\mint v^0(t,\bar x ,\tau,y,z)\phi(t,\bar x,\tau,y,z)\dzdy.
	 	\end{align*}
	 	The last equality we obtained by using the Riesz representation theorem.
	\end{proof}

    \begin{lemma}\label{lemma:h-1bound}
    Assume $f\in C_c^{\infty}((0,T)\times\Sigma;\mathcal{C}_{\sharp}^\infty(\overline{\T\times Y\times Z}))$ and satisfies 
    \begin{align}\label{eq:integral_f}
        \int_Z f\di{z}=0.
    \end{align}
    Let $f^\es(t,x):=f\left(t,\overline{x},\dfrac{t}{\es}, \dfrac{x}{\es},\dfrac{x_2}{\es^2}\right)$ then, there exists an $\es$-independent constant $C>0$,  such that
    \begin{align*}
       \| \frac{1}{\es^2} f^\es\|_{L^2(0,T;H^{1}(\omm)^*)}\leq C.
    \end{align*}
\end{lemma}
\begin{proof}
    The proof follows similar lines to \cite[Theorem 3.3]{allaire1996multiscale}. We have \eqref{eq:integral_f}, therefore, there exists $F\in C_c^{\infty}((0,T)\times\Sigma;C_{\sharp}(\overline{\T\times Y\times Z}))$ such that 
    \begin{align}\label{eq:derivative_of_F}
        \frac{d}{dz}F=f\quad\mbox{and}\quad   \int_ZF\di{z}=0.
    \end{align}
    Let $F^\es(t,x):=F\left(t,\overline{x},\dfrac{t}{\es}, \dfrac{x}{\es},\dfrac{x_2}{\es^2}\right)$, then
    \begin{align*}
        \frac{\p}{\p x_2}F^\es=\nabla_{\bar x} F^\es+\frac{1}{\es}\frac{\p}{\p y_2}F^\es+\frac{1}{\es^2}\frac{d}{dz}F^\es.
    \end{align*}
  From \eqref{eq:derivative_of_F}, we have
  \begin{align*}
      \int_Z\frac{\p}{\p y_2}F\di{z}=0
  \end{align*}
    therefore, there exist $\widetilde{F}\in C_c^{\infty}((0,T)\times\Sigma;C_{\sharp}(\overline{\T\times Y\times Z}))$ such that 
    \begin{align}\label{eq:derivative_of_F_widetilde}
        \frac{d}{dz}\widetilde{F}=\frac{\p}{\p y_2}F \quad \mbox{and }\quad    \int_Z\widetilde{F}\di{z}=0.
    \end{align}
     Let $\widetilde{F}^\es(t,x):=\widetilde{F}\left(t,\overline{x},\dfrac{t}{\es}, \dfrac{x}{\es},\dfrac{x_2}{\es^2}\right).$ By the chain rule and using \eqref{eq:derivative_of_F_widetilde}, we get
    \begin{align*}
        \frac{\p}{\p x_2}\widetilde{F}^\es=\nabla_{\bar x} \widetilde{F}^\es+\frac{1}{\es}\frac{\p}{\p y_2}\widetilde{F}^\es+\frac{1}{\es^2}\frac{d}{dz}\widetilde{F}^\es,
    \end{align*}
    and
    \begin{align*}
        \frac{1}{\es} \frac{\p}{\p y_2}F=\es \frac{\p}{\p x_2}\widetilde{F}^\es-\es \nabla_{\bar x} \widetilde{F}^\es -\frac{\p}{\p y_2}\widetilde{F}^\es.
    \end{align*}
Now, from \eqref{eq:derivative_of_F}, we have
\begin{align*}
\frac{1}{\es^2}f^\es&= \frac{\p}{\p x_2}F^\es-\nabla_{\bar x} F^\es - \frac{1}{\es}\frac{\p}{\p y_2}F^\es\\
&=\frac{\p}{\p x_2}F^\es-\nabla_{\bar x} F^\es -\Big(\es \frac{\p}{\p x_2}\widetilde{F}^\es-\es \nabla_{\bar x} \widetilde{F}^\es -\frac{\p}{\p y_2}\widetilde{F}^\es\Big)
\end{align*}
    Now, let $\psi\in L^2((0,T);H^1(\Omega_M^\es))$, with $\|\psi\|_{L^2((0,T);H^1(\Omega_M^\es))}\leq 1$,  then using integration by parts with respect to the $x_2$ variable, we get
    \begin{align*}
        \Big|\int_0^T\Big\langle \frac{1}{\es^2}f,\psi\Big\rangle \di{t}\Big|&=-\int_0^T\int_{\omm} F^\es \frac{\p}{\p x_2}\psi\di{x}\di{t}- \int_0^T\int_{\omm} \nabla_{\bar x} F^\es\psi\di{x}\di{t}\\
        &\quad+\es\int_0^T\int_{\omm} \widetilde{F}^\es \frac{\p}{\p x_2}\psi\di{x}\di{t}+\es \int_0^T\int_{\omm} \nabla_{\bar x} \widetilde{F}^\es \psi\di{x}\di{t}+ \int_0^T\int_{\omm} \frac{\p}{\p y_2}\widetilde{F}^\es \psi\di{x}\di{t}\\
        &\leq C.
    \end{align*}
    
\end{proof}

 \begin{proof}[Proof of \Cref{lemma:first_transmission_condition}]
        Let $\Phi(t,\bar x,y)\in (C_c^\infty ((0,T)\times\Sigma\times (Y\cup Y_L)))^2$ and extend $Y$ periodically in $y_2$ direction. From Theorem \ref{theorem:convergence_in_the_layer}, we have
        \begin{align*}
           \lim_{\es\rightarrow 0} \int_0^T\int_{\omm}\nabla u_M^\es\cdot \Phi(t,\bar x,\tfrac{x}{\es})\di{x}\di{t}=0
        \end{align*}
        Integration by parts with respect to the space variable, we get
        \begin{align*}
            0&=-\lim_{\es\rightarrow 0} \int_0^T\int_{\omm}u_M^\es \Big(\nabla\cdot \Phi(t,\bar x,\tfrac{x}{\es})+\frac{1}{\es}\nabla_y \cdot\Phi(t,\bar x,\tfrac{x}{\es})\Big)\di{x}\di{t}\\
            &\hspace{2cm}+\int_0^T\int_{\partial\omm} u_M^\es\Phi(t,\bar x,\tfrac{x}{\es})\cdot\nu\di{\sigma_x}\di{t}\\
            &=-\lim_{\es\rightarrow 0} \int_0^T\int_{\omm}u_M^\es \Big(\nabla\cdot \Phi(t,\bar x,\tfrac{x}{\es})+\frac{1}{\es}\nabla_y \cdot\Phi(t,\bar x,\tfrac{x}{\es})\Big)\di{x}\di{t}\\
            &\hspace{2cm}+\int_0^T\int_{L^\es} u_L^\es\Phi(t,\bar x,\tfrac{-1}{2},\tfrac{x_2}{\es})\cdot\nu\di{\sigma_x}\di{t}\\
            &=\mint u_M^0(t,x)\nabla_y\cdot \Phi(t,\bar x,y)\di{z}\di{y}\di{\tau}\di{\bar x}\di{t}\\
            &\hspace{2cm}-\int_0^T\int_{\Sigma}\int_{Y_L}u_L^0(t,\bar x)\Phi_1(t,\bar x, y_2)\di{y_2}\di{\bar x}\di{t}
        \end{align*}
        again doing integration by parts and by density argument, we get $u_L^0=u_M^0$ on $(0,T)\times\Sigma.$ A similar argument yields $u_R^0=u_M^0$ on $(0,T)\times\Sigma.$
    \end{proof}
    \section*{Acknowledgement}
This work is funded by the Deutsche Forschungsgemeinschaft (DFG, German Research Foundation) – Project-ID 258734477 – SFB 1173 and under Germany's Excellence Strategy – EXC-2047/2 – 390685813.

\bibliographystyle{abbrv}
	\bibliography{mybib}

\end{document}